\documentclass[11pt,a4paper,english]{article}

\newcommand{\LL}{\mathbb{L}_x}
\newcommand{\R}{\mathbb{R}}

\newcommand{\beq}{\begin{equation}}
\newcommand{\eeq}{\end{equation}}
\newcommand{\bea}{\begin{eqnarray}}
\newcommand{\eea}{\end{eqnarray}}
\newcommand{\beas}{\begin{eqnarray*}}
\newcommand{\eeas}{\end{eqnarray*}}

\def\cS{\mathcal{S}}
\def\RR{\mathbb{R}}

\usepackage[latin1]{inputenc}
\usepackage{babel}
\usepackage[centertags]{amsmath}
\usepackage{amsfonts}
\usepackage{amssymb}
\usepackage{amsthm}
\usepackage{authblk}
\usepackage{verbatim}
\usepackage{newlfont}
\usepackage{dsfont}
\usepackage{geometry}
\usepackage{mathrsfs}
\usepackage{color}
\usepackage{graphicx}
\usepackage{subcaption}
\usepackage{hyperref}

\newtheorem{theorem}{Theorem}[section]
\newtheorem{lemma}[theorem]{Lemma}
\newtheorem{coroll}[theorem]{Corollary}
\newtheorem{prop}[theorem]{Proposition}
\newtheorem{definition}[theorem]{Definition}

\newtheorem{remark}[theorem]{Remark}
\newtheorem{ass}[theorem]{Assumption}

\newcommand{\msc}[1]{\textbf{MSC2010 Classification:} #1.}
\newcommand{\jel}[1]{\textbf{JEL Classification:} #1.}
\newcommand{\keywords}[1]{\textbf{Key words:} #1.}

\def\vs{\vspace}
\def\hs{\hspace}

\def\EE{\mathsf E}
\def\PP{\mathsf P}
\def\CC{\mathcal C}
\def\DD{\mathcal D}
\def\eps{\varepsilon}
\def\cF{{\cal F}}

\def\cJ{{\cal J}}
\def\LL{\mathbb L}

\makeatletter  
\@addtoreset{equation}{section}
\def\theequation{\arabic{section}.\arabic{equation}}
\makeatother  
\begin{document}
\title{\textbf{
Optimal Entry to an Irreversible Investment Plan\\ with Non Convex Costs}}

\author{Tiziano De Angelis,\:\:Giorgio Ferrari,\:\:Randall Martyr, \\ John Moriarty}
\maketitle
\begin{abstract}
A problem of optimally purchasing electricity at a {\em real-}valued spot price (that is, with potentially negative cost)  
has been recently addressed in De Angelis, Ferrari and Moriarty (2015) [SIAM J.~Control Optim.~53(3)]. 
This problem can be considered one of irreversible investment with a cost functional which is non convex with respect to the control variable. In this paper we study the optimal entry into this investment plan. The optimal entry policy can have an irregular boundary arising from this non convexity, with a kinked shape.  
\end{abstract}
\msc{60G40, 93E20, 35R35, 49L20, 65C05}
\vspace{+4pt}

\noindent \jel{C61, D92, E22, Q41}
\vspace{+8pt}

\noindent\keywords{continuous-time inventory, optimal stopping, singular stochastic control, irreversible investment, Ornstein-Uhlenbeck price process} 
\section{Introduction}

In this paper we consider the question of optimal entry into a plan of irreversible investment with a cost functional which is non convex with respect to the control variable. The irreversible investment problem is that of \cite{DeAFeMo14}, in which the investor commits to delivering a unit of electricity to a consumer at a future random time $\Theta$ and may purchase and store electricity in real time at the stochastic (and potentially negative) spot price $(X_t)_{t \geq 0}$. In the optimal entry problem considered here, the consumer is willing to offer a single fixed initial payment $P_0$ in return for this commitment and the investor must choose a stopping time $\tau$ at which to accept the initial premium and enter the contract. If $\Theta \leq \tau$ then the investor's opportunity is lost and in this case no cashflows occur. If $\tau<\Theta$ then the inventory must be full at the time $\Theta$ of demand, any deficit being met by a less efficient method whose additional cost is represented by a convex factor $\Phi$ of the undersupply. The investor seeks to minimise the total expected costs, net of the initial premium $P_0$, by choosing $\tau$ optimally and by optimally filling the inventory from time $\tau$ onwards. 

Economic problems of optimal entry and exit under uncertain market prices have attracted significant interest. In the simplest formulation the timing of entry and/or exit is the only decision to be made and the planning horizon is infinite: see for example \cite{Dixit89} and \cite{McDonaldSiegel}, in which the market price is a geometric Brownian motion (GBM), and related models in \cite{DP} and \cite{Trigeorgis}. An extension of this problem to multiple types of economic activity is considered in \cite{BrekkeOksendal} and solved using stochastic calculus. In addition to the choice of entry / exit time, the decision problem may also depend on another control variable representing for instance investment or production capacity. For example in \cite{DZ00} the rate of production is modeled as a progressively measurable process whereas in \cite{PhamGuo} the production capacity is a process of bounded variation. In this case the problem is usually solved by applying the dynamic programming principle to obtain an associated Hamilton-Jacobi-Bellman (HJB) equation. If the planning horizon is finite then the optimal stopping and control strategies are time-dependent and given by suitable curves, see for example~\cite{DeAFe13}.

Typically, although not universally, the costs in the aforementioned problems are assumed to be convex with respect to the control variable. In addition to being reasonable in a wide range of problems, this assumption usually simplifies the mathematical analysis. In the present problem the underlying commodity is electricity, for which negative prices have been observed in several markets (see, e.g., \cite{GR} and \cite{LS}). The spot price is modelled by an Ornstein-Uhlenbeck process which is mean reverting and may take negative values and, as shown in \cite{DeAFeMo14}, this makes our control problem neither convex nor concave: to date such problems have received relatively little attention in the literature. In our setting the control variable represents the cumulative amount of electricity purchased by the investor in the spot market for storage. This control is assumed to be monotone, so that the sale of electricity back to the market is not possible, and also bounded to reflect the fact that the inventory used for storage has finite capacity. The investment problem falls into the class of singular stochastic control (SSC) problems (see \cite{Alvarez99}, \cite{KaratzasShreve84}, \cite{KaratzasShreve86}, among others).  

Borrowing ideas from \cite{PhamGuo}, we begin by decoupling the control (investment) problem from the stopping (entry) problem. The value function of this mixed stopping-then-control problem is shown to coincide with that of an appropriate optimal stopping problem over an infinite time-horizon whose gain function is the value function of the optimal investment problem with fixed entry time equal to zero. Unlike the situation in \cite{PhamGuo}, however, the gain function in the present paper is a function of two variables without an explicit representation. Indeed \cite{DeAFeMo14} identifies three regimes for the gain function, depending on the problem parameters, only two of which are solved rigorously: a {\em reflecting} regime, in which the control may be singularly continuous, and a {\em repelling} regime, in which the control is purely discontinuous. We therefore only address these two cases in this paper and leave the remaining open case for future work.

The optimal entry policies obtained below depend on the spot price and the inventory level and are described by suitable curves. 
On the one hand, for the reflecting case we prove that the optimal entry time is of a single threshold type as in \cite{DZ00} and \cite{PhamGuo}. On the other hand, the repelling case is interesting since it gives either a single threshold strategy or, alternatively, a complex optimal entry policy such that for any fixed value of the inventory level, the continuation region may be disconnected.

The paper is organised as follows. In Section \ref{ProblemFormulation} we set up the mixed irreversible investment-optimal entry problem, whose two-step formulation is then obtained in Section \ref{decouple}. Section \ref{sec:Timing} is devoted to the analysis of the optimal entry decision problem, with the repelling case studied separately in Section \ref{sec:chatgo}. Afterwards follows the conclusion, appendix, acknowledgements and references.

\section{Problem Formulation}
\label{ProblemFormulation}

We begin by recalling the optimal investment problem introduced in \cite{DeAFeMo14}.
Let $(\Omega,\mathcal{A},\PP)$ be a complete probability space, on which is defined a one-dimensional standard Brownian motion $(B_t)_{t\ge0}$.
We denote by $\mathbb{F}:=(\cF_t)_{t\ge0}$ the filtration generated by $(B_t)_{t\ge0}$ and augmented by $\PP$-null sets.
As in \cite{DeAFeMo14}, the spot price of electricity $X$ follows a standard time-homogeneous Ornstein-Uhlenbeck process with positive volatility $\sigma$, positive adjustment rate $\theta$ and positive asymptotic (or equilibrium) value $\mu$; i.e., $X^x$ is the unique strong solution of
\begin{align}
\label{OU}
dX^x_t= \theta(\mu-X_t^x)dt + \sigma dB_t, \quad\text{for $t>0$,\,\,\, with $X^x_0=x\in\RR$.}
\end{align}
\noindent Note that this model allows negative prices, which is consistent with the requirement to balance supply and demand in real time in electrical power systems and also consistent with the observed prices in several electricity spot markets (see, e.g., \cite{GR} and \cite{LS}).

We denote by $\Theta$ the random time of a consumer's demand for electricity. This is modelled as an $\mathcal{A}$-measurable positive random variable independent of $\mathbb{F}$ and distributed according to an exponential law with parameter $\lambda>0$, so that effectively the time of demand is completely unpredictable.
Note also that since $\Theta$ is independent of $\mathbb{F}$, the Brownian motion $(B_t)_{t\geq 0}$ remains a Brownian motion in the enlarged filtration $\mathbb{G}:=(\mathcal{G}_t)_{t\ge0}$, with $\mathcal{G}_t:=\mathcal{F}_t \vee \sigma(\{\Theta \leq s\}:\, s \leq t)$, under which $\Theta$ becomes a stopping time (see, e.g., Chapter 5, Section 6 of \cite{JYC}).

We will denote by $\tau$ any element of $\mathcal{T}$, the set of all $(\mathcal{F}_t)$-stopping times. At any $\tau$ the investor may enter the contract by accepting the initial premium $P_0$ and committing to deliver a unit of electricity at the time $\Theta$.
At any time during $[\tau, \Theta)$ electricity may be purchased in the spot market and stored, thus increasing the total inventory $C^{c,\nu} = (C^{c,\nu})_{t \ge 0}$, which is defined as
\begin{align}
C^{c,\nu}_t:=c + \nu_t\,,\qquad t\ge0.
\end{align}
Here $c\in [0,1]$ denotes the inventory at time zero and $\nu_t$ is the cumulative amount of electricity purchased up to time $t$. We specify the (convex) set of admissible investment strategies by requiring that $\nu \in \mathcal{S}^c_{\tau}$, where
\begin{eqnarray}
\label{admissiblecontrols}
\mathcal{S}^c_\tau \hspace{-0.2cm}&: = & \hspace{-0.2cm} \{\nu:\Omega \times \mathbb{R}_{+} \mapsto  \mathbb{R}_{+}, (\nu_{t}(\omega))_{t \geq 0}\mbox{ is nondecreasing,\,\,left-continuous,} \nonumber \\
&& \hspace{2cm} (\mathcal{F}_t)-\mbox{adapted, with} \,\,c + \nu_t\leq 1\,\,\,\forall t \geq 0, \nonumber \,\,\nu_\tau=0\,\,\,\,\PP-\mbox{a.s.}\}.
\end{eqnarray}
The amount of energy in the inventory is bounded above by 1 to reflect the investor's limited ability to store. The left continuity of $\nu$ ensures that any electricity purchased at time $\Theta$ is irrelevant for the optimisation. The requirement that $\nu$ be $(\mathcal{F}_t)$-adapted guarantees that all investment decisions are taken only on the basis of the price information available up to time $t$. The optimisation problem is given by 
\begin{align}\label{vfun00}
\inf_{\tau\ge0,\:
\nu\in\cS^c_{\tau}}\EE\Big[\Big(\int_{\tau}^{\Theta}{X^x_t}d\nu_t+X^{x}_\Theta\Phi(C^{y,\nu}_\Theta) - P_0\Big)\mathds{1}_{\{\tau < \Theta\}}\Big].
\end{align}
Here the first term represents expenditure in the spot market and the second is a penalty function:
if the inventory is not full at time $\Theta$ then it is filled by a less efficient method, so that the terminal spot price is weighted by a strictly convex function $\Phi$. We make
the following standing assumption:
\begin{ass}\label{ass:phi}
$\Phi: \mathbb{R} \mapsto \mathbb{R}_+$ lies in $C^2(\mathbb{R})$ and is decreasing and strictly convex in $[0,1]$ with $\Phi(1)=0$.
\end{ass}
\noindent For simplicity we assume that costs are discounted at the rate $r=0$. This involves no loss of generality since the independent random time of demand performs an effective discounting, as follows.
Recalling that $\Theta$ is independent of $\mathbb{F}$ and distributed according to an exponential law with parameter $\lambda>0$, Fubini's theorem gives that \eqref{vfun00} may be rewritten as
\begin{align}
\label{Vfun0}
V(x,c):=\inf_{\tau\ge0,\:
\nu\in\cS^c_{\tau}}\cJ_{x,c}(\tau,\nu)
\end{align}
with
\begin{align}\label{def:J}
\cJ_{x,c}(\tau,\nu):=\EE\Big[\int^{\infty}_\tau {\hs{-4pt}e^{-\lambda t}X^x_t\,d{\nu}_t}+\int_\tau^{\infty}\hs{-4pt}e^{-\lambda t}\lambda X^x_{t}\Phi(C^{c,\nu}_{t})dt-e^{-\lambda\tau}P_0\Big],
\end{align}
adopting the convention that on the set $\{\tau = +\infty\}$ we have $\int_{\tau}^{\infty}:=0$ and $e^{-\lambda\tau}:=0$. The discounting of costs may therefore be accomplished by appropriately increasing the exponential parameter $\lambda$.

\section{Decoupling the Problem and Background Material}
\label{decouple}

To deal with \eqref{Vfun0} we borrow arguments from \cite{PhamGuo} to show that the stopping (entry) problem can be split from the control (investment) problem, leading to a two-step formulation. We first briefly recall some results from \cite{DeAFeMo14}, where the control problem has the value function
\begin{align}\label{def:U}
U(x,c):=\inf_{\nu\in\cS^c_0}\cJ^0_{x,c}(\nu)
\end{align}
with 
\begin{align}\label{def:J0}
\cJ^0_{x,c}(\nu):=\EE\Big[ \int^{\infty}_0 {e^{-\lambda t}X^x_t\,d{\nu}_t} + \int_0^{\infty}e^{-\lambda s}\lambda X^x_{t}\Phi(C^{c,\nu}_{t})dt \Big]\,.
\end{align}
As was shown in \cite[Sec.~2]{DeAFeMo14}, the function 
\begin{align}
\label{def-k}
k(c):=\lambda+\theta+\lambda \,\Phi'(c), \qquad c\in\mathbb{R},
\end{align}
appears in an optimal stopping functional which may be associated with $U$. For convenience we let $\hat{c}\in\mathbb{R}$ denote the unique solution of $k(c)=0$ if it exists and write
\begin{align}\label{def:zeta}
\zeta(c):=\int_c^1{k(y)dy}=(\lambda + \theta)(1-c) - \lambda\Phi(c),\qquad c\in[0,1].
\end{align}
We formally introduce the variational problem associated with $U$:
\begin{align}
\label{HJB-U}
\max\{-\mathbb{L}_XU+ \lambda U-\lambda x \Phi(c),-U_c-x\}=0,\qquad\text{on $\mathbb{R} \times (0,1),$}
\end{align}
where $\mathbb{L}_X$ is the second order differential operator associated to the infinitesimal generator of $X$:
\begin{align}\label{def:LX}
\mathbb{L}_{X}f\,(x):=\frac{1}{2}\sigma^2 f''(x) + \theta(\mu - x)f'(x),\quad\text{for $f\in C^2_b(\RR)$ and $x\in\RR$.}
\end{align}
According to standard theory on control problems we define the inaction set for problem \eqref{def:U} by
\begin{align}\label{def:C}
\CC:=\{(x,c)\in \mathbb{R}\times[0,1]\,:\,U_c(x,c)>-x \}.
\end{align}
The non convexity of functional \eqref{def:J0} with respect to the control variable $\nu$, which arises due to the real-valued factor $X^x_t$, places it outside the standard existing literature on SSC problems. We therefore collect here the solutions proved in Sections 2 and 3 of \cite{DeAFeMo14}.
\begin{prop}\label{prop:backmat}
We have $|U(x,c)|\le C(1+|x|)$ for $(x,c)\in\RR\times[0,1]$ and a suitable constant $C>0$. Moreover the following holds
\begin{itemize}
\item[  i)] If $\hat{c}<0$ (i.e.~$k(\,\cdot\,)>0$ in $[0,1]$), then $U\in C^{2,1}(\mathbb{R}\times[0,1])$ and it is a classical solution of \eqref{HJB-U}. The inaction set \eqref{def:C} is given by
\begin{align}
\CC=\{(x,c)\in \mathbb{R}\times[0,1]\,:\,x>\beta_*(c) \}
\end{align}
for some function $\beta_*\in C^1([0,1])$ which is decreasing and dominated from above by $x_0(c)\wedge\hat{x}_0(c)$, $c\in[0,1]$, with
\begin{align}\label{eq:x0xhat}
x_0(c):=-\theta\mu\Phi'(c)/k(c)\quad\text{and}\quad \hat{x}_0(c):=\theta\mu/k(c),
\end{align}
(cf.~\cite[Prop.~2.5 and Thm.~2.8]{DeAFeMo14}). For $c\in[0,1]$ the optimal control is given by
\begin{align}
\label{optcontr00}
\nu_t^*=\left[g_*\left(\inf_{0 \leq s \leq t} X^x_s\right)-c\right]^+, \qquad t>0, \quad
\nu_0^*= 0,
\end{align}
with $g_*(x):=\beta_*^{-1}(x)$, $x\in(\beta_*(1),\beta_*(0))$, and $g_* \equiv 0$ on $[\beta_*(0),\infty)$, $g_* \equiv 1$ on $(-\infty,\beta_*(0)]$.
\item[ ii)]  If $\hat{c}>1$ (i.e.~$k(\,\cdot\,)<0$ in $[0,1]$), then $U\in W^{2,1,\infty}_{loc}(\mathbb{R}\times[0,1])$ and it solves \eqref{HJB-U} in the a.e.~sense. The inaction set \eqref{def:C} is given by
\begin{align}
\CC=\{(x,c)\in \mathbb{R}\times[0,1]\,:\,x<\gamma_*(c) \}
\end{align}
with suitable $\gamma_*\in C^1([0,1])$, decreasing and bounded from below by $\tilde{x}(c)\vee\overline{x}_0(c)$, $c\in[0,1]$, with
\begin{align}\label{eq:x0xtild}
\overline{x}_0(c):= \theta\mu\Phi(c)/\zeta(c)\quad\text{and}\quad\tilde{x}(c):=\theta\mu(1-c)/\zeta(c),
\end{align}
(cf.~\cite[Thm.~3.1 and Prop.~3.4]{DeAFeMo14}). Moreover $U(x,c)=x(1-c)$ for $x\ge\gamma_*(c)$, $c\in[0,1]$, and for any $c\in[0,1]$ the optimal control is given by (cf.~\cite[Thm.~3.5]{DeAFeMo14})
\begin{align}
\label{op-contr01}
\nu^*_t:=\left\{
\begin{array}{ll}
0, & t\le \tau_*,\\
(1-c), & t>\tau_*
\end{array}
\right.
\end{align}
with $\tau_*:=\inf\big\{t\ge 0\,:\,X^x_t\ge \gamma_*(c)\big\}$.
\end{itemize}
\end{prop}

We now perform the decoupling into two sub-problems, one of control and one of stopping.

\begin{prop}\label{decoupling}
If $\hat{c}<0$ or $\hat{c}>1$ then the value function $V$ of \eqref{Vfun0} can be equivalently rewritten as
\begin{align}\label{dec01}
V(x,c)=\inf_{\tau\ge0}\EE\Big[e^{-\lambda \tau}\Big(U(X^x_\tau,c)-P_0\Big)\Big],
\end{align}
with the convention $e^{-\lambda \tau}(U(X^x_\tau,c)-P_0):=\liminf_{t \uparrow \infty} e^{-\lambda t}(U(X^x_t,c)-P_0) =0$ on $\{\tau=\infty\}$.
\end{prop}
\begin{proof}
Let us set
\begin{align}\label{eq:w}
w(x,c):=\inf_{\tau\ge0}\EE\Big[e^{-\lambda \tau}\Big(U(X^x_\tau,c)-P_0\Big)\Big],\qquad\text{for $(x,c)\in\RR\times[0,1]$.}
\end{align}
Thanks to the results of Proposition \ref{prop:backmat} we can apply It\^o's formula to $U$, in the classical sense in case $i)$ and in its generalised version (cf.~\cite[Ch.~8, Sec.~VIII.4, Thm.~4.1]{FlemingSoner}) in case $ii)$. In particular for an arbitrary stopping time $\tau$, an arbitrary admissible control $\nu\in\cS^c_\tau$ and with $\tau_n:=\tau\vee n$, $n\in\mathbb{N}$ we get 
\begin{align}\label{itoU}
\EE\Big[e^{-\lambda\tau_n}U(X^x_{\tau_n},C^{c,\nu}_{\tau_n})\Big] = {} & \EE\Big[ e^{-\lambda\tau}U(X^x_{\tau},c)\Big]+\EE\bigg[\int^{\tau_n}_\tau{e^{-\lambda t}\big(\mathbb{L}_XU-\lambda U\big)(X^x_t,C^{c,\nu}_t)dt}\bigg] \nonumber \\
&+\EE\bigg[\int_{\tau}^{\tau_n}{e^{-\lambda t}U_c(X^x_t,C^{c,\nu}_t)d\nu^{cont}_t}\bigg]\nonumber\\
&+\EE\bigg[\sum_{\tau\le t<\tau_n}e^{-\lambda t}\Big(U(X^x_t,C^{c,\nu}_{t+})-U(X^x_t,C^{c,\nu}_{t})\Big)\bigg],
\end{align}
where we have used standard localisation techniques to remove the martingale term, and decomposed the control into its continuous and jump parts, i.e.~$d\nu_t=d\nu_t^{cont}+\Delta\nu_t$, with $\Delta\nu_t:=\nu_{t+}-\nu_t$. Since $U$ solves the HJB equation \eqref{HJB-U} it is now easy to prove (cf.~for instance \cite[Thm.~2.8]{DeAFeMo14}) that, in the limit as $n\to\infty$, one has
\begin{align}\label{DDP1}
\EE\Big[e^{-\lambda\tau}U(X^x_{\tau},c)\Big]\le \EE\Big[\int^\infty_{\tau}{e^{-\lambda t}\lambda X^x_t\Phi(C^{c,\nu}_t)dt}+\int^{\infty}_\tau{ e^{-\lambda t}X^x_t d\nu_t}\Big],
\end{align}
and therefore
\begin{align}\label{deco01}
\EE \Big[e^{-\lambda\tau}\big(U(X^x_\tau,c)-P_0\big)\Big]\le\EE\Big[\int^\infty_{\tau}{e^{-\lambda t}\lambda X^x_t\Phi(C^{c,\nu}_t)dt}+\int^{\infty}_\tau{ e^{-\lambda t}X^x_t d\nu_t}-e^{-\lambda\tau}P_0\Big],
\end{align}
for an arbitrary stopping time $\tau$ and an arbitrary control $\nu\in\cS^c_\tau$. Hence by taking the infimum over all possible stopping times and over all $\nu\in\cS^c_\tau$, \eqref{Vfun0}, \eqref{eq:w} and \eqref{deco01} give $w(x,c)\le V(x,c)$.

To prove that equality holds, let us fix an arbitrary stopping time $\tau$. In case $i)$ of Proposition \ref{prop:backmat}, one can pick a control $\nu^\tau\in\cS^c_\tau$ of the form
\begin{align}\label{eq:nu-i}
\text{$\nu^\tau_t=0$ for $t\le\tau$ \,\, and \,\, $\nu^\tau_t=\nu^*_t$ for $t>\tau$}
\end{align}
with $\nu^*$ as in \eqref{optcontr00}, to obtain equality in \eqref{DDP1} and hence in \eqref{deco01}. In case $ii)$ instead we define $\sigma^*_\tau:=\inf\{t\ge\tau\,:\,X^x_t\ge\gamma_*(c)\}$ and pick $\nu^\tau\in\cS^c_\tau$ of the form
\begin{align}\label{eq:nu-ii}
\text{$\nu^\tau_t=0$ for $t\le\sigma^*_\tau$ \,\, and \,\, $\nu^\tau_t=1-c$ for $t>\sigma^*_\tau$}
\end{align}
to have again equality in \eqref{DDP1} and hence in \eqref{deco01}. Now taking the infimum over all $\tau$ we find $w(x,c)\ge V(x,c)$.

To complete the proof we need to prove the last claim; that is, $\liminf_{t \uparrow \infty} e^{-\lambda t}(U(X^x_t,c)-P_0) =0$ a.s. It suffices to show that $\liminf_{t \uparrow \infty} e^{-\lambda t}|U(X^x_t,c)-P_0| =0$ a.s. To this end recall that $|U(x,c)|\le C(1+|x|)$, for $(x,c)\in\RR\times[0,1]$ and a suitable constant $C>0$ (cf.\ Proposition \ref{prop:backmat}), and then apply Lemma \ref{lem:convX} in Appendix \ref{factsOU}.
\end{proof}

\begin{remark}\label{rem:obvssol}
	The optimal stopping problems \eqref{dec01} depend only parametrically on the inventory level $c$ (the case $c=1$ is trivial as $U(\,\cdot\,,1)=0$ on $\RR$ and the optimal strategy is to stop at once for all initial points $x\in\RR$).
\end{remark}

It is worth noting that we were able to perform a very simple proof of the decoupling knowing the structure of the optimal control for problem \eqref{def:U}. In wider generality one could obtain a proof based on an application of the \emph{Dynamic Programming Principle} although in that case it is well known that some delicate measurability issues should be addressed as well (see \cite{PhamGuo}, Appendix A).
Although each of the optimal stopping problems \eqref{dec01} is for a one-dimensional diffusion over an infinite time horizon, standard methods find only limited application since no explicit expression is available for their gain function $U(x,c)-P_0$. 

In the next section we show that the cases $\hat{c}<0$ and $\hat{c}>1$, which are the regimes solved rigorously in \cite{DeAFeMo14}, have substantially different optimal entry policies.
To conclude with the background we prove a useful concavity result.
\begin{lemma}\label{lem:concave}
The maps $x\mapsto U(x,c)$ and $x\mapsto V(x,c)$ are concave for fixed $c\in[0,1]$.
\end{lemma}
\begin{proof}
We begin by observing that $X^{px+(1-p)y}_t=pX^x_t+(1-p)X^y_t$ for all $t\ge0$ and any $p\in(0,1)$. Hence \eqref{def:J0} gives
\begin{align*}
\cJ^0_{px+(1-p)y,c}(\nu)=p\cJ^0_{x,c}(\nu)+(1-p)\cJ^0_{y,c}(\nu)\ge p U(x,c)+(1-p)U(y,c), \qquad \forall\nu\in\cS^c_0
\end{align*}
and therefore taking the infimum over all admissible $\nu$ we easily find $U(px+(1-p)y,c)\ge p U(x,c)+(1-p)U(y,c)$ as claimed.

For $V$ we argue in a similar way and use concavity of $U(\,\cdot\,,c)$ as follows: let $\tau\ge0$ be an arbitrary stopping time, then
\begin{align*}
\EE\Big[e^{-\lambda \tau}\Big(U(X^{px+(1-p)y}_\tau,c)-P_0\Big)\Big]
= {} & \EE\Big[e^{-\lambda \tau}\Big(U(pX^{x}_\tau+(1-p)X^y_\tau,c)-P_0\Big)\Big]\\[+4pt]
\ge {} & \EE\Big[e^{-\lambda \tau}\Big(pU(X^{x}_\tau,c)+(1-p)U(X^{y}_\tau,c)-P_0\Big)\Big]\\[+4pt]
= {} & p \times \EE\Big[e^{-\lambda \tau}\Big(U(X^{x}_\tau,c)-P_0\Big)\Big] \\
& +(1-p) \times \EE\Big[e^{-\lambda \tau}\Big(U(X^{y}_\tau,c)-P_0\Big)\Big]\\[+4pt]
\ge {} & \,\,p\,V(x,c)+(1-p)V(y,c).
\end{align*}
We conclude the proof by taking the infimum over all stopping times $\tau\ge0$.

\end{proof}

\section{Timing the Entry Decision}
\label{sec:Timing}
\label{sec:heuristics}

We first examine the optimal entry policy via a standard argument based on exit times from small intervals of $\RR$. 
An application of Dynkin's formula gives that the instantaneous `cost of continuation' in our optimal entry problem is given by the function 
\begin{equation}\label{eq:inst}
\mathcal{L}(x,c)+\lambda P_0:=(\mathbb{L}_X-\lambda)(U-P_0)(x,c).
\end{equation}
In the case $\hat c < 0$, which is covered in Section \ref{sec:chatneg}, the function \eqref{eq:inst} is monotone decreasing (see \eqref{signLU} below). Since problem \eqref{vfun00} is one of minimisation, it is never optimal to stop at points $(x,c)\in\RR\times[0,1]$ such that $\mathcal{L}(x,c)+\lambda P_0<0$; an easy comparison argument then shows there is a unique lower threshold that determines the optimal stopping rule in this case.

When $\hat{c}>1$ the picture is more complex. The function \eqref{eq:inst} is decreasing and continuous everywhere except at a single point where it has a positive jump (cf. Proposition \ref{prop:discont} below) and so can change sign twice. The comparison argument now becomes more subtle: continuation should not be optimal when the function \eqref{eq:inst} is positive in a `large neighbourhood containing the initial value $x$'. Indeed it will turn out in Section \ref{sec:chatgo} that there are multiple possible optimal stopping regimes depending on parameter values. In particular the continuation region of the optimal stopping problem may be disconnected, which 
is unusual in the literature on optimal entry problems. The resulting optimal entry region has a kinked shape (Figure \ref{fig:examples}). The jump in the function \eqref{eq:inst} arises from the `bang-bang' nature of the optimal investment plan when $\hat c > 1$, and so this may be understood as causing the unusual shape of the optimal entry boundary.

\begin{figure}[htbp]
\begin{center}
\includegraphics[scale=0.5]{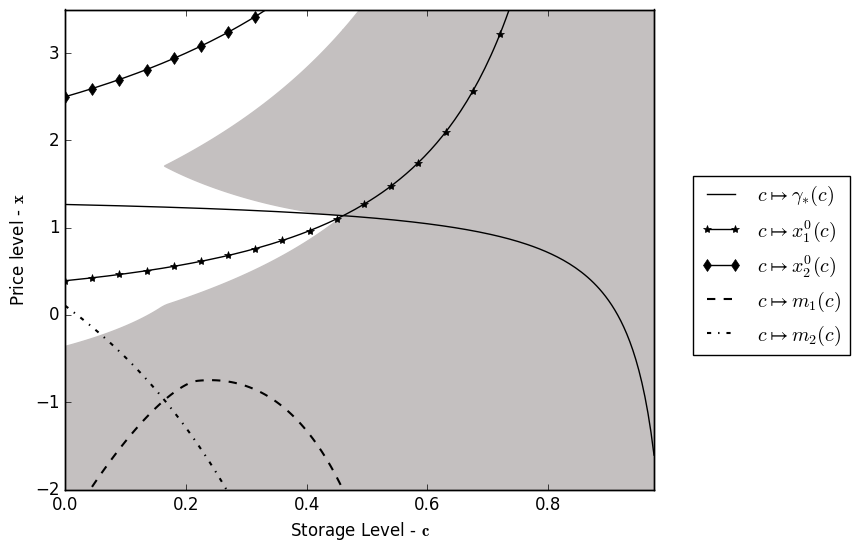}
\caption{An indicative example of an optimal entry region (shaded) when $\hat c > 1$, together with the functions $\gamma_*$ and $x^0_1$, $x^0_2$ (introduced in Prop.~\ref{prop:discont} below). The functions $m_{1}$ and $m_{2}$ (not drawn to scale) are important determinants for the presence of the kinked shape (see Remark~\ref{rem:red} below). This plot was generated using $\mu = 1$, $\theta = 1$, $\sigma = 3$, $\lambda = 1$, $P_{0} = 4$ and $\Phi(c) = 2.2(1 - c) + 8(1 - c)^{2}$.}
\label{fig:examples}
\end{center}
\end{figure}

\subsection{The case $\hat{c}<0$}\label{sec:chatneg}
Let us now assume that $\hat{c}<0$, i.e.~$k(c)>0$ for all $c\in[0,1]$ (cf.~\eqref{def-k}). We first recall from Section 2.2 of \cite{DeAFeMo14} that in this case
\begin{align}\label{eq:U&u}
U(x,c)=x(1-c)-\int_c^1{u(x;y)dy},\qquad\text{for $(x,c)\in\RR\times[0,1]$,}
\end{align}
where $u$ is the value function of an associated optimal stopping problem with (cf. Sections 2.1 and 2.2 of \cite{DeAFeMo14})
\begin{align}
\label{eq:u01}& \text{$(i)$~~$u(\,\cdot\,,c)\in W^{2,\infty}_{loc}(\RR)$ for any $c\in[0,1]$}\\[+3pt]
\label{eq:u02}& \text{$(ii)$~$u(x,c)> 0$ for $x>\beta_*(c)$ and $u(x,c)=0$ for $x\le\beta_*(c)$, $c\in[0,1],$}
\end{align}
and with $\beta_*$ given as in Proposition \ref{prop:backmat}-$i)$. Moreover, defining
\begin{align}\label{def:G}
G(x,c) := \frac{\mu(k(c)-\theta)}{\lambda} + \frac{k(c)(x-\mu)}{\lambda+\theta},
\end{align}
and recalling $\phi_{\lambda}$ from Definition~\ref{def:FundamentalSolutions}, $u$ is expressed analytically as
\begin{align}\label{def:u-analyt}
u(x,c)=\left\{\begin{array}{ll}
G(x,c)-\frac{G({\beta_*(c)},c)}{\phi_{\lambda}({\beta_*(c)})}\phi_{\lambda}(x), & x > {\beta_*(c)} \\ [+4pt]
0, & x \leq {\beta_*(c)}\end{array}
\right.
\end{align}
for $c\in[0,1]$, and it solves the variational problem
\begin{align}
&\big(\LL_X-\lambda \big)u(x,c)=\theta\mu-k(c)x & x>\beta_*(c),\,c\in[0,1]\\[+3pt]
&\big(\LL_X-\lambda \big)u(x,c)=0 & x\le \beta_*(c),\,c\in[0,1]\\[+3pt]
& u(\beta_*(c),c)=u_x(\beta_*(c),c)=0 & c\in[0,1].
\end{align}
By the regularity of $u$ and dominated convergence we have
\begin{equation}\label{eq:LXU}
(\mathbb{L}_X-\lambda )U(x,c)=(1-c)(\theta\mu-(\lambda+\theta)x)-\int_c^1(\mathbb{L}_X-\lambda)u(x;y)dy
\end{equation}
for $(x,c)\in\RR\times[0,1]$.

As is usual, for each $c\in[0,1]$ we define the continuation region $\CC^c_V$ and stopping region $\DD^c_V$ for the optimal stopping problem \eqref{dec01} as
\begin{align}\label{def:CD}
\CC^c_V=\{x\in\RR\,:\,V(x,c) < U(x,c)-P_0\}\,,\quad\DD^c_V=\{x\in\RR\,:\,V(x,c) = U(x,c)-P_0\}.
\end{align}
With the aim of characterising the geometry of $\CC^c_V$ and $\DD^c_V$ we start by providing some preliminary results on $U-P_0$ that will help to formulate an appropriate free-boundary problem for $V$.
\begin{prop}
\label{prop-sign-LU}
For any given $c\in[0,1]$, there exists a unique $x^0(c)\in\mathbb{R}$ such that
\begin{align}
\label{signLU}
\big(\mathbb{L}_X-\lambda\big)\big(U(x,c)-P_0)\left\{
\begin{array}{ll}
<0 & \text{for $x>x^0(c)$}\\[+3pt]
=0 & \text{for $x=x^0(c)$}\\[+3pt]
>0 & \text{for $x<x^0(c)$}
\end{array}
\right.
\end{align}
\end{prop}
\noindent We refer to Appendix \ref{someproofs} for the proof of the previous proposition.

As discussed at the beginning of Section \ref{sec:heuristics}, it is never optimal in problem \eqref{dec01} to stop in $(x^0(c),\infty)$, $c\in[0,1]$, for $x^0(c)$ as in Proposition \ref{prop-sign-LU}, i.e.
\begin{align}\label{eq:inclC}
(x^0(c),\infty)\subseteq \CC^c_V\qquad\text{for $c\in[0,1]$,}
\end{align}
and consequently 
\begin{align}\label{eq:inclD}
\DD^c_V \subset [-\infty,x^0(c)]\qquad\text{for $c\in[0,1]$.}
\end{align}
Hence we conjecture that the optimal stopping strategy should be of single threshold type.
In what follows we aim at finding $\ell_*(c)$, $c\in[0,1]$, such that $\DD^c_V=[-\infty,\ell_*(c)]$ and
\begin{align}\label{opt-st02}
\tau^*(x,c) = \inf \{ t \geq 0: X^x_t \leq \ell_*(c)\}
\end{align}
is optimal for $V(x,c)$ in \eqref{dec01} with $(x,c)\in\RR\times[0,1]$. The methodology adopted in \cite[Sec.~2.1]{DeAFeMo14} does not apply directly to this problem due to the semi-explicit expression of the gain function $U-P_0$.

\subsubsection{Formulation of Auxiliary Optimal Stopping Problems}
\label{sec:aux1}

To work out the optimal boundary $\ell_*$ we will introduce auxiliary optimal stopping problems and employ a \emph{guess-and-verify} approach in two frameworks with differing technical issues. We first observe that since $U$ is a classical solution of \eqref{HJB-U}, an application of Dynkin's formula to \eqref{dec01} provides a lower bound for $V$, that is
\begin{align}\label{aux01}
V(x,c) \geq U(x,c)-P_0 + \Gamma(x,c),\qquad (x,c)\in\RR\times[0,1],
\end{align}
with
\begin{align}\label{def-gamma}
\Gamma(x,c):=\inf_{\tau \geq 0}
 \EE\bigg[ \int_0^{\tau}e^{-\lambda s}\big(\lambda P_0-\lambda X^x_s \Phi(c)\big)ds
\bigg]\qquad (x,c)\in\RR\times[0,1].
\end{align}
On the other hand, for $(x,c)\in\RR\times[0,1]$ fixed, set $\sigma^*_\beta:=\inf\{t\ge0\,:\,X^x_t\le\beta_*(c)\}$ with $\beta_*$ as in Proposition \ref{prop:backmat}, then for arbitrary stopping time $\tau$ one also obtains
\begin{align}\label{eq:Vupbdd00}
\EE&\Big[e^{-\lambda(\tau\wedge\sigma^*_\beta)}\Big(U(X^x_{\tau\wedge\sigma^*_\beta},c)-P_0\Big)\Big]\\
&=U(x,c)-P_0+\EE\bigg[ \int_0^{\tau \wedge \sigma^*_\beta}e^{-\lambda s}\big(\lambda P_0-\lambda X^x_s \Phi(c)\big)ds
\bigg]\nonumber
\end{align}
by using the fact that $U$ solves \eqref{HJB-U} and Dynkin's formula. We can now obtain an upper bound for $V$ by setting
\begin{align}\label{def-gammab}
\Gamma_\beta(x,c):=\inf_{\tau \geq 0}\EE\bigg[ \int_0^{\tau \wedge \sigma^*_\beta}e^{-\lambda s}\big(\lambda P_0-\lambda X^x_s \Phi(c)\big)ds
\bigg],\quad(x,c)\in\RR\times[0,1],
\end{align}
so that taking the infimum over all $\tau$ in \eqref{eq:Vupbdd00} one obtains
\begin{align}\label{aux02}
V(x,c) &\leq U(x,c)-P_0 +  \Gamma_\beta(x,c)\qquad(x,c)\in\RR\times[0,1].
\end{align}

It turns out that \eqref{aux01} and \eqref{aux02} allow us to find a simple characterisation of the optimal boundary $\ell_*$ and of the function $V$ in some cases. Let us first observe that $0 \geq \Gamma_\beta(x,c) \geq \Gamma (x,c)$ for all $(x,c) \in \mathbb{R} \times [0,1]$. Defining for each fixed $c \in [0,1]$ the stopping regions
\beas
\DD^c_{\Gamma}=\{x\in\RR\,:\, \Gamma (x,c)=0\} \qquad\text{and}\qquad
\DD^c_{\Gamma_\beta}=\{x\in\RR\,:\, \Gamma_\beta (x,c)=0\}
\eeas
it is easy to see that $\DD^c_{\Gamma} \subset \DD^c_{\Gamma_\beta}$. Moreover, by the monotonicity of $x\mapsto X^x_\cdot$ it is not hard to verify that $x \mapsto \Gamma(x,c)$ and $x\mapsto\Gamma_\beta(x,c)$ are decreasing. Hence we again expect optimal stopping strategies of threshold type, i.e.
\begin{align}\label{def:Dg}
\DD^c_{\Gamma}=\{x\in\RR\,:\, x\le\alpha^*_1(c)\} \qquad\text{and}\qquad
\DD^c_{\Gamma_\beta}=\{x\in\RR\,:\, x\le \alpha^*_2(c)\}
\end{align}
for $c\in[0,1]$ and for suitable functions $\alpha^*_i(\,\cdot\,)$, $i=1,2$ to be determined.

Assume for now that $\alpha^*_1$ and $\alpha^*_2$ are indeed optimal, then we must have
\begin{align}\label{eq:bineq}
\alpha^*_1(c) \leq \ell_*(c) \leq \alpha_2^*(c)\qquad\text{for $c\in[0,1]$}.
\end{align}
Indeed, for all $(x,c) \in \RR\times[0,1]$ we have $\DD^c_{\Gamma} \subset \DD^c_{V}$ since $\Gamma (x,c) \leq V(x,c)-U(x,c)+P_0 \leq 0$, and $\DD^c_{V} \subset \DD^c_{\Gamma_\beta}$ since $V(x,c)-U(x,c)+P_0 \leq \Gamma_\beta (x,c) \le 0$. Notice also that since the optimisation problem in \eqref{def-gammab} is the same as the one in \eqref{def-gamma} except that in the former the observation is stopped when $X$ hits $\beta_*$, we must have
\begin{equation}\label{eq:ThresholdStrategyProof}
\alpha^*_2(c)=\beta_*(c)\vee \alpha_1^*(c) \quad \text{for} \enskip c\in[0,1].
\end{equation}
Thus for each $c\in[0,1]$ we can now consider two cases:
\begin{enumerate}
\item if $\alpha_1^*(c) > \beta_*(c)$ we have $\Gamma (x,c)= \Gamma_\beta(x,c)=\big(V-U+P_0\big)(x,c)$ for $x\in\RR$ and $\ell_*(c) = \alpha_1^*(c)$,\label{gamma-c02}
\item if $\alpha_1^*(c) \le \beta_*(c)$ we have
$\alpha^*_2(c)=\beta_*(c)$, implying that $\ell_*(c) \leq \beta_*(c)$.\label{gamma-c01}
\end{enumerate}

Both $1.$\ and $2.$\ above need to be studied in order to obtain a complete characterisation of $\ell_*$, however we note that case $1.$\ is particularly interesting as it identifies $V$ and $\ell_*$ with $\Gamma+U-P_0$ and $\alpha^*_1$, respectively. As we will clarify in what follows, solving problem \eqref{def-gamma} turns out to be theoretically simpler and computationally less demanding than dealing directly with problem \eqref{dec01}.

\subsubsection{Solution of the Auxiliary Optimal Stopping Problems}
\label{sec:solaux}

To make our claims rigorous we start by analysing problem \eqref{def-gamma}. This is accomplished by largely relying on arguments already employed in \cite[Sec.~2.1]{DeAFeMo14} and therefore we omit proofs here whenever a precise reference can be provided. Moreover, the majority of the proofs of new results are provided in Appendix \ref{someproofs} to simplify the exposition. Let us now introduce two functions $\phi_{\lambda}$ and $\psi_\lambda$ which feature frequently below.
\begin{definition}\label{def:FundamentalSolutions}
	Let $\phi_{\lambda} : \RR\to\RR^+$ and $\psi_\lambda:\RR\to\RR^+$ denote respectively the decreasing and increasing fundamental solutions of the differential equation $\mathbb{L}_Xf=\lambda f$ on $\RR$ (see Appendix \ref{factsOU} for details).
\end{definition}

In problem \eqref{def-gamma} we conjecture an optimal stopping time of the form
\begin{align}
\tau_\alpha(x,c) := \inf \{ t \geq 0\,:\, X^x_t \leq \alpha(c)\}
\end{align}
for $(x,c)\in\RR\times[0,1]$ and $\alpha$ to be determined. Under this conjecture $\Gamma$ should be found in the class of functions of the form
\beq\label{gammaeq}
\Gamma^{\alpha}(x,c)=\left\{\begin{array}{ll}
\displaystyle \EE\bigg[ \int_0^{\tau_{\alpha}}e^{-\lambda s}\lambda\big(P_0- X^x_s \Phi(c)\big)ds\bigg], & x > \alpha(c) \\[+3pt]
0, & x \leq \alpha(c)\end{array}
\right.
\eeq
for each $c\in[0,1]$. Now, repeating the same arguments of proof of \cite[Thm.~2.1]{DeAFeMo14} we obtain
\begin{lemma}\label{lem:candgamma}
One has
\begin{align}\label{eq:analitg}
\Gamma^{\alpha}(x,c)=\left\{\begin{array}{ll}
\big(P_0-\hat{G}(x,c)\big)-\big(
P_0-\hat{G}(\alpha(c),c)
\big)\frac{\phi_{\lambda}(x)}{\phi_{\lambda}(\alpha(c))}, & x > \alpha(c) \\[+5pt]
0, & x \leq \alpha(c)\end{array}
\right.
\end{align}
for each $c\in[0,1]$, with
\begin{align}\label{def-G-hat}
\hat{G}(x,c):=\mu\Phi(c)+(x-\mu)\tfrac{\lambda\Phi(c)}{\lambda+\theta}\qquad(x,c)\in\RR\times[0,1].
\end{align}
\end{lemma}

To single out the candidate optimal boundary we impose the so-called {\em smooth fit} condition, i.e.~$\tfrac{d}{dx}\Gamma^\alpha(\alpha(c),c)=0$ for every $c\in[0,1]$. This amounts to finding $\alpha^*$ such that
\begin{align}\label{smfit03}
-\tfrac{\lambda\Phi(c)}{\lambda+\theta}+
\big(\hat{G}(\alpha^*(c),c)-P_0
\big)\tfrac{\phi'_{\lambda}(\alpha^*(c))}{\phi_{\lambda}(\alpha^*(c))}=0\quad\text{for $c\in[0,1]$}.
\end{align}
\begin{prop}\label{prop:smfit01}
For $c\in[0,1]$ define
\begin{align}\label{x-dag}
x^\dagger_0(c):=\mu+\big(P_0-\mu\Phi(c)\big)\tfrac{(\lambda+\theta)}{\lambda\Phi(c)}.
\end{align}
For each $c\in[0,1]$ there exists a unique solution $\alpha^*(c)\in(-\infty,x^\dagger_0(c))$ of \eqref{smfit03}. Moreover $\alpha^*\in C^1([0,1))$ and it is strictly increasing with $\lim_{c\to1}\alpha^*(c)=+\infty$.
\end{prop}
\noindent For the proof of Proposition \ref{prop:smfit01} we refer to Appendix \ref{someproofs}.

To complete the characterisation of $\alpha^*$ and $\Gamma^{\alpha^*}$ we now find an alternative upper bound for $\alpha^*$ that will guarantee $\big(\LL_X \Gamma^{\alpha^*}-\lambda\Gamma^{\alpha^*}\big)(x,c)\ge -\lambda(P_0-x\Phi(c))$ for $(x,c)\in\RR\times[0,1]$. Again, the proof of the following result may be found in Appendix \ref{someproofs}.
\begin{prop}\label{prop:upbdda}
For all $c\in[0,1]$ we have $\alpha^*(c)\le P_0/\Phi(c)$ with $\alpha^*$ as in Proposition \ref{prop:smfit01}.
\end{prop}

With the aim of formulating a variational problem for $\Gamma^{\alpha^*}$ we observe that $\tfrac{d^2}{dx^2}\Gamma^{\alpha^*}(x,c)>0$ for $x>\alpha^*(c)$, $c\in[0,1]$ by \eqref{eq:analitg}, convexity of $\phi_\lambda$ and the fact that $\hat{G}(\alpha^*(c),c)-P_0<0$. Hence $\Gamma^{\alpha^*}\ge0$ on $\RR\times[0,1]$. It is not hard to verify by direct calculation from \eqref{eq:analitg} and the above results that for all $c\in[0,1]$ the couple $\big(\Gamma^{\alpha^*}(\,\cdot\,,c),\alpha^*(c)\big)$ solves the free-boundary problem
\begin{align}
\label{free00}&\big(\LL_X-\lambda\big)\Gamma^{\alpha^*}(x,c)=-\lambda(P_0-x\Phi(c)) & x>{\alpha^*}(c),\\[+4pt]
\label{free01}&\big(\LL_X-\lambda\big)\Gamma^{\alpha^*}(x,c)>-\lambda(P_0-x\Phi(c)) & x < \alpha^*(c),\\[+4pt]
\label{free02}&\Gamma^{\alpha^*}(\alpha^*(c),c)=\Gamma^{\alpha^*}_x(\alpha^*(c),c)=0,\quad \Gamma^{\alpha^*}(x,c)\ge0&x\in\RR
\end{align}
and $\Gamma^{\alpha^*}(\,\cdot\,,c)\in W^{2,\infty}_{loc}(\RR)$. Following now the same arguments as in the proof of \cite[Thm.~2.1]{DeAFeMo14}, which is based on an application of the It\^o-Tanaka formula and \eqref{free00}--\eqref{free02}, we can verify our guess and prove the following theorem (the details are omitted).

\begin{theorem}\label{thm:verif01}
The boundary $\alpha^*$ of Proposition \ref{prop:smfit01} is optimal for \eqref{def-gamma} in the sense that $\alpha^*=\alpha^*_1$ with $\alpha^*_1$ as in \eqref{def:Dg},
\begin{align}\label{eq:sigmastar}
\tau^*_\alpha=\inf\{t \geq 0\,:\,X^x_t \leq {\alpha_*(c)}\}
\end{align}
is an optimal stopping time and $\Gamma^{\alpha^*}\equiv \Gamma$ (cf.~\eqref{def-gamma}).
\end{theorem}
\vspace{+5pt}

\subsubsection{Solution of the Original Optimal Stopping Problem \eqref{dec01}}
\label{sec:noaux}

In Theorem \ref{thm:verif01} we have fully characterised $\alpha^*_1$ and $\Gamma$ thus also $\alpha^*_2$ and $\Gamma_\beta$ (cf.~\eqref{def-gammab}, \eqref{def:Dg} and \eqref{eq:ThresholdStrategyProof}). Moreover we have found that $\alpha^*_1(\,\cdot\,)$ is strictly increasing on $[0,1)$. On the other hand, $\beta_*(\,\cdot\,)$ is a strictly decreasing function (cf.~Proposition \ref{prop:backmat}-$i)$), hence there exists at most one $c_* \in (0,1)$ such that
\begin{align}\label{beta*alpha*}
\text{$\beta_*(c) > \alpha^*_1(c)$ for $c \in (0,c_*)$ and $\beta_*(c) \le \alpha^*_1(c)$ for $c \in [c_*,1)$}.
\end{align}
As already mentioned, it may be possible to provide examples where such a value $c_*$ does not exist $(0,1)$ and $\alpha^*_1(c)>\beta_*(c)$ for all $c\in [0,1]$. In those cases, as discussed in Section \ref{sec:aux1}, one has $\ell_*=\alpha^*_1$ and $V=U-P_0+\Gamma$ and problem \eqref{dec01} is fully solved. 
Therefore to provide a complete analysis of problem \eqref{dec01} we must consider the case when $c_*$ exists in $(0,1)$. From now on we make the following assumption.
\begin{ass}\label{ass:c*}
There exists a unique $c_*\in(0,1)$ such that \eqref{beta*alpha*} holds.
\end{ass}
As a consequence of the analysis in Section \ref{sec:solaux} we have the next simple corollary.
\begin{coroll}\label{cor-Vgamma}
For all $c \in [c_*,1)$ it holds $V(x,c)=(\Gamma+U-P_0)(x,c)$, $x\in\RR$ and $\ell_*(c)=\alpha^*_1(c)$, with $\Gamma$ and $\alpha^*_1$ as in Theorem \ref{thm:verif01}.
\end{coroll}

It remains to characterise $\ell_*$ in the interval $[0,c_*)$ in which we have $\ell_*(c) \leq \beta_*(c)$. This is done in Theorem~\ref{thm:UniqueSolutionSmoothFit-1}, whose proof requires other technical results which are cited here and proved in the appendix. Fix $c\in[0,c_*)$, let $\ell(c)\in\mathbb{R}$ be a candidate boundary and define the stopping time
$\tau_\ell(x,c):=\inf\big\{t\ge0\,:\,X^x_t\le \ell(c)\big\}$ for $x\in\RR$.
Again to simplify notation we set $\tau_\ell=\tau_\ell(x,c)$ when no confusion may arise. It is now natural to associate to $\ell(c)$ a candidate value function
\begin{align}\label{def-Vl}
V^\ell(x,c):=\EE\left[e^{-\lambda\tau_\ell}\Big(U(X^x_{\tau_\ell},c)-P_0\Big)\right],
\end{align}
whose analytical expression is provided in the next lemma.
\begin{lemma}\label{lemma12}
For $c\in[0,c_*)$ we have
\begin{align}\label{Vl-repr}
V^\ell(x,c)=\left\{
\begin{array}{ll}
(U(\ell(c),c)-P_0)
\frac{\phi_{\lambda}(x)}{\phi_{\lambda}(\ell(c))}, & x>\ell(c) \\[+4pt]
U(x,c)-P_0, & x \leq \ell(c)
\end{array}
\right.
\end{align}
\end{lemma}

The candidate boundary $\ell_*$, whose optimality will be subsequently verified, is found by imposing the smooth fit condition, i.e. \begin{align}
\label{smfit04}
(U(\ell_*(c),c)-P_0)
\frac{\phi'_{\lambda}(\ell_*(c))}{\phi_{\lambda}(\ell_*(c))}
=U_x(\ell_*(c),c), \qquad c\in[0,1].
\end{align}

\begin{prop}
\label{prop-smfitV}
For any $c\in[0,c_*)$ there exists at least one solution $\ell_*(c)\in(-\infty, x^0(c))$ of \eqref{smfit04} with $x^0(c)$ as in Proposition \ref{prop-sign-LU}.
\end{prop}

\begin{remark}\label{rem:Vl}
	A couple of remarks before we proceed.
	
	\vspace{+4pt}
	
	i. The analytical representation \eqref{Vl-repr} in fact holds for all $c\in[0,1]$ and it must coincide with \eqref{eq:analitg} for $c\in[c_*,1]$. Furthermore, the optimal boundary $\alpha^*_1$ found in Section \ref{sec:solaux} by solving \eqref{smfit03} must also solve \eqref{smfit04} for all $c\in[c_*,1]$ since $\alpha^*_1=\ell_*$ on that set. This equivalence can be verified by comparing numerical solutions to \eqref{smfit03} and \eqref{smfit04}. 
		Finding a numerical solution to \eqref{smfit04} for $c\in[0,c_*)$ (if it exists) is computationally more demanding than solving \eqref{smfit03}, however, because of the absence of an explicit expression of the function $U$.
		
	\vspace{+4pt}
	
	ii. It is important to observe that the proof of Proposition \ref{prop-smfitV} does not use that $c\in[0,c_*)$ and in fact it holds for $c\in[0,1]$. However, arguing as in Section \ref{sec:solaux} we managed to obtain further regularity properties of the optimal boundary in $[c_*,1]$ and its uniqueness. We shall see in what follows that uniqueness can be retrieved also in $c\in[0,c_*)$ but it requires a deeper analysis.
\end{remark}

Now that the existence of at least one candidate optimal boundary $\ell_*$ has been established, for the purpose of performing a verification argument we would also like to establish that for arbitrary $c\in[0,c_*)$ we have $V^{\ell_*}(x,c)\le U(x,c)-P_0$, $x\in\RR$. This is verified in the following proposition (whose proof is collected in appendix).

\begin{prop}\label{prop-VxUx}
For $c\in[0,c_*)$ and for any $\ell_*$ solving \eqref{smfit04} it holds $V^{\ell_*}(x,c)\le U(x,c)-P_0$, $x\in\RR$.
\end{prop}

Finally we provide a verification theorem establishing the optimality of our candidate boundary $\ell_*$ and, as a by-product, also implying uniqueness of the solution to \eqref{smfit04}.
\begin{theorem}\label{thm:UniqueSolutionSmoothFit-1}
\label{verificationOS-case2}
There exists a unique solution of \eqref{smfit04} in $(-\infty,x^0(\bar{c})]$. This solution is the optimal boundary of problem \eqref{dec01} in the sense that $V^{\ell_*}=V$ on $\RR\times[0,1)$ (cf.~\eqref{Vl-repr}) and the stopping time
\begin{align}\label{eq:opt-stt}
\tau^*:=\tau^*_\ell(x,c)=\inf\{t\ge0\,:\,X^x_t\le\ell_*(c)\}
\end{align}
is optimal in \eqref{dec01} for all $(x,c)\in\RR\times[0,1)$.
\end{theorem}
\begin{proof}
For $c\in[c_*,1)$ the proof was provided in Section \ref{sec:solaux} recalling that $\ell_*=\alpha^*_1$ on $[c_*,1)$ and $V=U-P_0+\Gamma$ on $\RR\times[c_*,1)$ (cf.~\eqref{def-gamma}, Remark \ref{rem:Vl}). For $c\in[0,c_*)$ we split the proof into two parts.
\vspace{+4pt}

$1.$ \emph{Optimality.} Fix $\bar{c}\in[0,c_*)$. Here we prove that if $\ell_*(\bar{c})$ is any solution of \eqref{smfit04} then $V^{\ell_*}(\,\cdot\,,\bar{c})= V(\,\cdot\,,\bar{c})$ on $\RR$ (cf.~\eqref{dec01} and \eqref{Vl-repr}).

First we note that $V^{\ell_*}(\,\cdot\,,\bar{c})\ge V(\,\cdot\,,\bar{c})$ on $\RR$ by \eqref{dec01} and \eqref{def-Vl}.
To obtain the reverse inequality we will rely on It\^o-Tanaka's formula. Observe that $V^{\ell_*}(\,\cdot\,,\bar{c})\in C^1(\mathbb{R})$ by \eqref{Vl-repr} and \eqref{smfit04}, and $V^{\ell_*}_{xx}(\,\cdot\,,\bar{c})$ is continuous on $\mathbb{R}\setminus \big\{\ell_*(\bar{c})\big\}$ and bounded at the boundary $\ell_*(\bar{c})$. Moreover from \eqref{Vl-repr} we get \begin{align}
\label{eq:Vell1}&\big(\mathbb{L}_X-\lambda\big)V^{\ell_*}(x,\bar{c})=0 & \text{for $x>\ell_*(\bar{c})$}\\[+4pt]  \label{eq:Vell2}&\big(\mathbb{L}_X-\lambda\big)V^{\ell_*}(x,\bar{c})=\big(\mathbb{L}_X-\lambda\big)(U-P_0)(x,\bar{c})>0 & \text{for $x\le \ell_*(\bar{c})$}
\end{align}
where the inequality in \eqref{eq:Vell2} holds by \eqref{signLU} since $\ell_*(\bar{c})\le x^0(\bar{c})$ (cf.~Proposition \ref{prop-smfitV}). An application of It\^o-Tanaka's formula (see~\cite{KS}, Chapter 3, Problem 6.24, p.~215), \eqref{eq:Vell1}, \eqref{eq:Vell2} and Proposition \ref{prop-VxUx} give
\begin{align}\label{verif09}
V^{\ell_*}(x,\bar{c}) & = \EE\left[e^{-\lambda(\tau\wedge\tau_R)}V^{\ell_*}\big(X^x_{\tau\wedge\tau_R},\bar{c}\big)-
\int_0^{\tau\wedge\tau_R}{e^{-rt}}\big(\mathbb{L}_X-\lambda\big)V^{\ell_*}(X^x_t,\bar{c})dt\right]\\
& \le
\EE\left[e^{-\lambda(\tau\wedge\tau_R)}\Big(U\big(X^x_{\tau\wedge\tau_R},\bar{c}\big)-P_0\Big)\right]\nonumber
\end{align}
with $\tau$ an arbitrary stopping time and $\tau_R:=\inf\big\{t\ge0\,:\,|X^x_t| \ge R\big\}$, $R>0$. We now pass to the limit as $R\to\infty$ and recall that $|U(x,\bar{c})|\le C(1+|x|)$ (cf.~Proposition \ref{prop:backmat}) and that $\big\{e^{-\lambda \tau_R}|X^x_{\tau_R}|\,,\,R>0\big\}$ is a uniformly integrable family (cf.~Lemma \ref{lem:uiX} in Appendix \ref{factsOU}). Then in the limit we use the dominated convergence theorem and the fact that
\begin{align*}
\lim_{R\to\infty}e^{-\lambda (\tau\wedge\tau_R)}X^x_{ \tau\wedge\tau_R }= e^{-\lambda
	\tau}X^x_{\tau},\qquad\PP-a.s.
\end{align*}
to obtain $V^{\ell_*}(\,\cdot\,,\bar{c})\le V(\,\cdot\,,\bar{c})$ on $\RR$ by the arbitrariness of $\tau$, hence $V^{\ell_*}(\,\cdot\,,\bar{c})=V(\,\cdot\,,\bar{c})$ on $\RR$ and optimality of $\ell_*(\bar{c})$ follows.
\vspace{+4pt}

$2.$ \emph{Uniqueness.} Here we prove the uniqueness of the solution of \eqref{smfit04} via probabilistic arguments similar to those employed for the first time in \cite{Pe05}. Let $\bar{c}\in[0,c_*)$ be fixed and, arguing by contradiction, let us assume that there exists another solution $\ell'(\bar{c})\neq\ell_*(\bar{c})$ of \eqref{smfit04} with $\ell'(\bar{c})\le x^0(\bar{c})$. Then by \eqref{dec01} and \eqref{def-Vl} it follows that 
\begin{align}\label{ineq}
V^{\ell'}(\,\cdot\,,\bar{c})\ge V(\,\cdot\,,\bar{c})=V^{\ell_*}(\,\cdot\,,\bar{c})\qquad\text{ on $\RR$,}
\end{align}
$V^{\ell'}(\,\cdot\,,\bar{c})\in C^1(\RR)$ and $V^{\ell'}_{xx}(\,\cdot\,,\bar{c})\in L^\infty_{loc}(\RR)$ by the same arguments as in $1.$\ above. By construction $V^{\ell'}$ solves \eqref{eq:Vell1} and \eqref{eq:Vell2} with $\ell_*$ replaced by $\ell'$.

Assume for example that $\ell'(\bar{c})< \ell_*(\bar{c})$, take $x<\ell'(\bar{c})$ and set $\sigma^*_{\ell}:=\inf\big\{t\ge0\,:\,X^x_t\ge \ell_*(\bar{c})\big\}$, then an application of It\^o-Tanaka's formula gives (up to a localisation argument as in $1.$\ above)
\begin{align}\label{uniq01}
\EE  \left[
e^{-\lambda\sigma_{\ell}^*} V^{\ell'} \big( X^x_{\sigma^*_{\ell}}, \bar{c} \big)
\right]& = V^{\ell'}(x,\bar{c})+\EE\Big[\int^{\sigma^*_{\ell}}_0{e^{-\lambda t}\big(\mathbb{L}_X-\lambda\big)
V^{\ell'}\big(X^x_t,\bar{c}\big)dt}\Big]\\
& = V^{\ell'}(x,\bar{c})+\EE\Big[\int^{\sigma^*_{\ell}}_0{e^{-\lambda t}\big(\mathbb{L}_X-\lambda\big)
\big(U\big(X^x_t,\bar{c}\big)-P_0\big)\mathds{1}_{\{X^x_t<\ell'(\bar{c})\}}dt}\Big]\nonumber
\end{align}
and
\begin{align}\label{uniq02}
\EE \left[
e^{-\lambda\sigma^*_{\ell}} V \big( X^x_{\sigma^*_{\ell}}, \bar{c} \big)
\right]& = V(x,\bar{c})+\EE\Big[\int^{\sigma^*_{\ell}}_0{e^{-\lambda t}\big(\mathbb{L}_X-\lambda\big)
\big(U\big(X^x_t,\bar{c}\big)-P_0\big)dt}\Big].
\end{align}
Recall that $V^{\ell'}(X^x_{\sigma^*_{\ell}},\bar{c})\ge V(X^x_{\sigma^*_{\ell}},\bar{c})$ by \eqref{ineq} and that for $x<\ell'(\bar{c})\le \ell_*(\bar{c})$ one has $V(x,\bar{c})=V^{\ell'}(x,\bar{c})=U(x,\bar{c})-P_0$, hence subtracting \eqref{uniq02} from \eqref{uniq01} we get
\begin{align}\label{eq:contr1}
-\EE\left[\int^{\sigma^*_{\ell}}_0{e^{-\lambda t}\big(\mathbb{L}_X-\lambda\big)
\big(U\big(X^x_t,\bar{c}\big)-P_0\big)\mathds{1}_{\{\ell'(\bar{c})<X^x_t<\ell_*(\bar{c})\}}dt}\right] \ge 0.
\end{align}
By the continuity of paths of $X^x$ we must have $\sigma^*_\ell>0$, $\PP$-a.s.~and since the law of $X$ is absolutely continuous with respect to the Lebesgue measure we also have $\PP\big(\{\ell'(\bar{c})<X^x_t<\ell_*(\bar{c})\}\big)>0$ for all $t > 0$. Therefore \eqref{eq:contr1} and \eqref{eq:Vell2} lead to a contradiction and we conclude that $\ell'(\bar{c})\ge \ell_*(\bar{c})$.

Let us now assume that $\ell'(\bar{c})> \ell_*(\bar{c})$ and take $x\in\big(\ell_*(\bar{c}),\ell'(\bar{c})\big)$. We recall the stopping time $\tau^*$ of \eqref{eq:opt-stt} and again we use It\^o-Tanaka's formula to obtain
\begin{align}\label{uniq03}
\EE\left[e^{-\lambda\tau^*}V\big(X^x_{\tau^*},\bar{c}\big)\right]=V(x,\bar{c})
\end{align}
and
\begin{align}\label{uniq04}
\EE \left[
e^{-\lambda\tau^*} V^{\ell'} \big( X^x_{\tau^*}, \bar{c} \big)
\right]=V^{\ell'}(x,\bar{c})+\EE\bigg[\int^{\tau^*}_0{e^{-\lambda t}\big(\mathbb{L}_X-\lambda\big)
\big(U\big(X^x_t,\bar{c}\big)-P_0\big)\mathds{1}_{\{X^x_t<\ell'(\bar{c})\}}dt}\bigg]
\end{align}
Now, we have $V(x,\bar{c})\le V^{\ell'}(x,\bar{c})$ by \eqref{ineq} and $V^{\ell'} \big( X^x_{\tau^*}, \bar{c} \big)=V\big(X^x_{\tau^*},\bar{c}\big)=U(\ell_*(\bar{c}),\bar{c})-P_0$, $\PP$-a.s.~by construction, since $\ell'(\bar{c})>\ell_*(\bar{c})$ and $X$ is positively recurrent (cf.~Appendix \ref{factsOU}). Therefore subtracting \eqref{uniq03} from \eqref{uniq04} gives
\begin{align}\label{eq:contr2}
\EE\Big[\int^{\tau^*}_0{e^{-\lambda t}\big(\mathbb{L}_X-\lambda\big)
\big(U\big(X^x_t,\bar{c}\big)-P_0\big)\mathds{1}_{\{\ell_*(\bar{c})<X^x_t<\ell'(\bar{c})\}}dt}\Big]\le0.
\end{align}
Arguments analogous to those following \eqref{eq:contr1} can be applied to \eqref{eq:contr2} to find a contradiction. Then we have $\ell'(\bar{c})= \ell_*(\bar{c})$ and by the arbitrariness of $\bar{c}$ the first claim of the theorem follows.
\end{proof}

\begin{remark}
The arguments developed in this section hold for all $c\in[0,1]$. The reduction of \eqref{dec01} to the auxiliary problem of Section \ref{sec:aux1} is not necessary to provide an algebraic equation for the optimal boundary. Nonetheless, it seems convenient to resort to the auxiliary problem whenever possible due to its analytical and computational tractability. In contrast to Section \ref{sec:solaux}, here we cannot establish either the monotonicity or continuity of the optimal boundary $\ell_*$. 
\end{remark}

\section{The Case $\hat{c}>1$}
\label{sec:chatgo}
 
In what follows we assume that $\hat{c}>1$, i.e.~$k(c)<0$ for all $c\in[0,1]$. As pointed out in Proposition \ref{prop:backmat}-$ii)$ the solution of the control problem in this setting substantially departs from the one obtained for $\hat{c}<0$. Both the value function and the optimal control exhibit a structure that is fundamentally different, and we recall here some results from \cite[Sec.~3]{DeAFeMo14}.

The function $U$ has the following analytical representation:
\begin{align}\label{eq:Uanalyt2}
U(x,c)=\left\{
\begin{array}{ll}
\hs{-6pt}\displaystyle \tfrac{\psi_{\lambda}(x)}{\psi_{\lambda}(\gamma_*(c))}\hs{-2pt}\left[\gamma_*(c)(1\hs{-2pt}-\hs{-2pt}c)\hs{-2pt} -\hs{-2pt} \lambda\,\Phi(c)
\big(\tfrac{\gamma_*(c)-\mu}{\lambda+\theta}\hs{-1pt}+\hs{-1pt}\tfrac{\mu}{\lambda}\big)\right] 
\hs{-2pt}+\hs{-2pt}\lambda\,\Phi(c)\hs{-2pt}
\left[\tfrac{x-\mu}{\lambda+\theta}\hs{-1pt}+\hs{-1pt}\tfrac{\mu}
{\lambda}\right], & \text{for $x<{\gamma_*(c)}$} \\[+12pt]
\hs{-6pt} \displaystyle x(1-c), & \text{for $x\ge \gamma_*(c)$}
\end{array}
\right.
\end{align}
with $\gamma_*$ as in Proposition \ref{prop:backmat}-$ii)$. In this setting $U$ is less regular than the one for the case of $\hat{c}<0$, in fact here we only have $U(\,\cdot\,,c)\in W^{2,\infty}_{loc}(\RR)$ for all $c\in[0,1]$ (cf.~Proposition \ref{prop:backmat}-$ii)$) and hence we expect $x\mapsto\mathcal{L}(x,c)+\lambda P_0:=(\mathbb{L}_X-\lambda)(U-P_0)(x,c)$ to have a discontinuity at the optimal boundary $\gamma_*(c)$. For $c\in[0,1]$ we define
\begin{align}\label{def:deltaL}
\Delta^{\mathcal{L}}(x,c):=\mathcal{L}(x+,c)-\mathcal{L}(x-,c),\qquad x\in\RR,
\end{align}
where $\mathcal{L}(x+,c)$ denotes the right limit of $\mathcal{L}(\,\cdot\,,c)$ at $x$ and $\mathcal{L}(x-,c)$ its left limit.
\begin{prop}\label{prop:discont}
For each $c\in[0,1)$ the map $x\mapsto\mathcal{L}(x,c)+\lambda P_0$ is $C^\infty$ and strictly decreasing on $(-\infty,\gamma_*(c))$ and on $(\gamma_*(c),+\infty)$ whereas
\begin{align}
\Delta^{\mathcal{L}}(\gamma_*(c),c)=(1-c)\big[\theta\mu-(\lambda+\theta)\gamma_*(c)\big]+\lambda\gamma_*(c)\Phi(c)>0.
\end{align}

Moreover, define
\begin{equation}
\label{twozeros}
x^0_1(c):= \frac{P_0}{\Phi(c)}\quad\text{and}\quad x^0_2(c):= \frac{\theta\mu(1-c) + \lambda P_0}{(\lambda + \theta)(1-c)},\quad c\in[0,1);
\end{equation}
then for each $c\in[0,1)$ there are three possible settings, that is
\begin{enumerate}
	\item $\gamma_*(c) \le x^0_1(c)$ hence $\mathcal{L}(x,c)+\lambda P_0> 0$ if and only if $x< x^0_2(c)$;
	\item $\gamma_*(c) \ge x^0_2(c)$ hence $\mathcal{L}(x,c)+\lambda P_0>0$ if and only if $x< x^0_1(c)$;
	\item $x^0_1(c) < \gamma_*(c) < x^0_2(c)$ hence $\mathcal{L}(x,c)+\lambda P_0>0$ if and only if $x\in (-\infty,x^0_1(c))\cup(\gamma_*(c),x^0_2(c))$.
\end{enumerate}
\end{prop}
\begin{proof}
The first claim follows by \eqref{eq:Uanalyt2} and the sign of $\Delta^\mathcal{L}(\gamma_*(c),c)$ may be verified by recalling that $\gamma_*(c)\ge \tilde{x}(c)$ (cf.~Proposition \ref{prop:backmat}-$ii)$). Checking $1$, $2$ and $3$ is matter of simple algebra.
\end{proof}
We may use Proposition \ref{prop:discont} to expand the discussion in Section \ref{sec:heuristics}. In particular, from the first and second parts we see that if either $\gamma_*(c) \ge x^0_2(c)$ or $\gamma_*(c) \le x^0_1(c)$ then the optimal stopping strategy must be of single threshold type. On the other hand, for $x^0_1(c) < \gamma_*(c) < x^0_2(c)$, as discussed in Section \ref{sec:heuristics}, there are two possible shapes for the continuation set. This is setting for the preliminary discussion which follows.

If the size of the interval $(\gamma_*(c),x^0_2(c))$ is ``small'' and/or the absolute value of $\mathcal{L}(x,c)+\lambda P_0$ in $(\gamma_*(c),x^0_2(c))$ is ``small'' compared to its absolute value in $(x^0_1(c),\gamma_*(c))\cup(x^0_2(c),+\infty)$ then, although continuation incurs a positive cost when the process is in the interval $(\gamma_*(c),x^0_2(c))$, the expected reward from subsequently entering the neighbouring intervals (where $\mathcal{L}(x,c)+\lambda P_0<0$) is sufficiently large that continuation may nevertheless be optimal in $(\gamma_*(c),x^0_2(c))$ so that there is a single lower optimal stopping boundary, which lies below $x^0_1(c)$ (see Figures~\ref{fig:examples} and \ref{fig:IG-Plot-1}).

\begin{figure}[htbp]
	\centering
	\begin{subfigure}[b]{0.45\textwidth}
		\includegraphics[width=\textwidth,height=0.175\textheight]{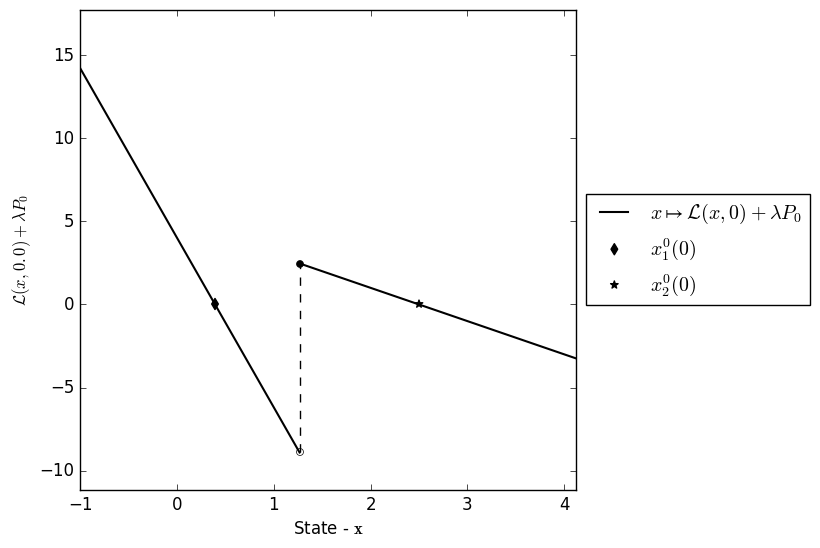}
		\caption{Illustration when $c = 0$}
		\label{fig:IG-Plot-1}
	\end{subfigure}
	~ 
	\begin{subfigure}[b]{0.45\textwidth}
		\includegraphics[width=\textwidth,height=0.175\textheight]{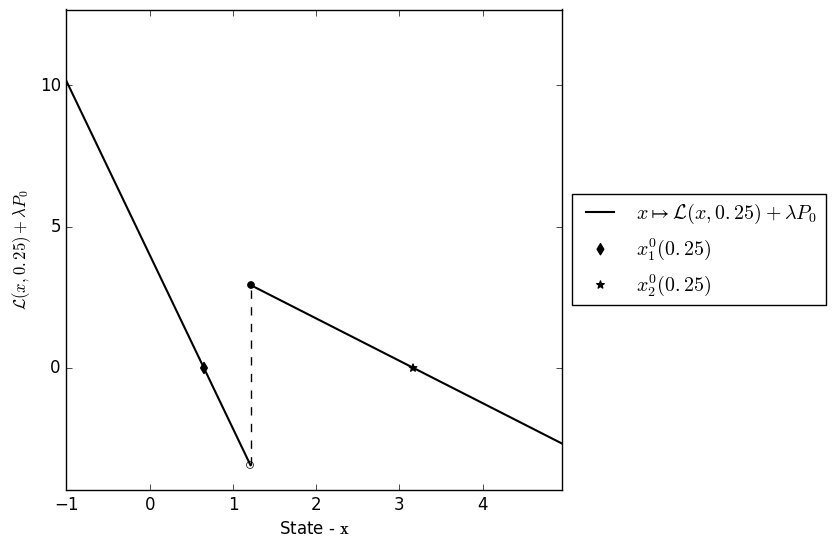}
		\caption{Illustration when $c = 0.25$}
		\label{fig:IG-Plot-2}
	\end{subfigure}
	\caption{The function $x \mapsto \mathcal{L}(x,c)+\lambda P_0$ changes sign in both plots but, with the visual aid of Figure~\ref{fig:examples}, the stopping region is connected in (a) and is disconnected in (b).}
	\label{fig:IG-Plots}
\end{figure}

If the size of $(\gamma_*(c),x^0_2(c))$ is ``big'' and/or the absolute value of $\mathcal{L}(x,c)+\lambda P_0$ in $(\gamma_*(c),x^0_2(c))$ is ``big'' compared to its absolute value in $(x^0_1(c),\gamma_*(c))\cup(x^0_2(c),+\infty)$ then we may find a portion of the stopping set below $x^0_1(c)$ and another portion inside the interval $(\gamma_*(c),x^0_2(c))$. In this case the loss incurred by continuation inside a certain subset of $(\gamma_*(c),x^0_2(c))$ may be too great to be mitigated by the expected benefit of subsequent entry into the profitable neighbouring intervals and it becomes optimal to stop at once. In the third case of Proposition \ref{prop:discont}, the continuation and stopping regions may therefore be disconnected sets (see Figures~\ref{fig:examples} and \ref{fig:IG-Plot-2}).

To make this discussion rigorous let us now recall $\CC_V^c$ and $\DD_V^c$ from \eqref{def:CD}. Note that for any fixed $c\in[0,1)$ and arbitrary stopping time $\tau$ the map $x\mapsto \EE[e^{-\lambda\tau}\big(U(X^x_\tau,c)-P_0\big)]$ is continuous, hence $x\mapsto V(x,c)$ is upper semicontinuous (being the infimum of continuous functions). Recall that $X$ is positively recurrent and therefore it hits any point of $\RR$ in finite time with probability one (see Appendix \ref{factsOU} for details). Hence according to standard  optimal stopping theory, if $\DD_V^c\neq\emptyset$ the first entry time of $X$ in $\DD_V^c$ is an optimal stopping time (cf.~e.g.~\cite[Ch.~1, Sec.~2, Corollary 2.9]{Pes-Shir}).
\begin{prop}\label{prop:struc-stop}
Let $c\in[0,1)$ be fixed. Then
\begin{itemize}
\item[$~~i)$] if $\gamma_*(c) \ge x^0_2(c)$, there exists $\ell_*(c)\in(-\infty,x^0_1(c))$ such that $\DD_V^c=(-\infty,\ell_*(c)]$ and $\tau_*=\inf\{t\ge 0\,:\,X^x_t\le \ell_*(c)\}$ is optimal in \eqref{dec01}
\item[$~ii)$] if $\gamma_*(c) \le x^0_1(c)$, there exists $\ell_*(c)\in(-\infty,x^0_2(c))$ such that $\DD_V^c=(-\infty,\ell_*(c)]$ and $\tau_*=\inf\{t\ge 0\,:\,X^x_t\le \ell_*(c)\}$ is optimal in \eqref{dec01}
\item[$iii)$] if $x^0_1(c) < \gamma_*(c) < x^0_2(c)$, there exists $\ell^{(1)}_*(c)\in(-\infty,x^0_1(c))$ such that $\DD_V^c\cap(-\infty,\gamma_*(c)]=(-\infty,\ell^{(1)}_*(c)]$. Moreover, either $(a)$: $\DD_V^c\cap[\gamma_*(c),\infty)=\emptyset$ and $\tau_*=\inf\{t\ge 0\,:\,X^x_t\le \ell^{(1)}_*(c)\}$ is optimal in \eqref{dec01}, or $(b)$: there exist $\ell_*^{(2)}(c)\le\ell_*^{(3)}(c)\le x^0_{2}(c)$ such that $\DD_V^c\cap[\gamma_*(c),\infty)=[\ell_*^{(2)}(c),\ell_*^{(3)}(c)]$ (with the convention that if $\ell_*^{(2)}(c)=\ell_*^{(3)}(c)=:\ell_*(c)$ then $\DD_V^c\cap[\gamma_*(c),\infty)=\{\ell_*(c)\}$) and the stopping time
\begin{align}\label{eq:ell*2}
\tau^{(II)}_*:=\inf\{t\ge0\,:\,X^x_t\le\ell^{(1)}_*(c)\:\:\text{or}\:\:X^x_t\in[\ell_*^{(2)}(c),\ell_*^{(3)}(c)]\}
\end{align}
is optimal in \eqref{dec01}.
\end{itemize}
\end{prop}
\begin{proof}
We provide a detailed proof only for $iii)$ as the other claims follow by analogous arguments. Let us fix $c\in[0,1)$ and assume $x^0_1(c) < \gamma_*(c) < x^0_2(c)$.
\vspace{+4pt}

Step $1.$ \enskip We start by proving that $\DD_V^c\neq\emptyset$. By localisation and an application of It\^o's formula in its generalised version (cf.~\cite[Ch.~8]{FlemingSoner}) to \eqref{dec01} and recalling Proposition \ref{prop:discont} we get
\begin{align}\label{set00}
V(x,c)=U(x,c)-P_0+\inf_\tau\Big[\int_0^\tau{e^{-\lambda t}\Big(\lambda P_0+\mathcal{L}(X^x_t,c)\Big)dt}\Big]\quad\text{for $x\in\RR$}.
\end{align}
Arguing by contradiction we assume that $\DD_V^c=\emptyset$ and hence the optimum in \eqref{set00} is obtained by formally setting $\tau=+\infty$. Moreover by recalling that $U$ solves \eqref{HJB-U} we observe that $\mathcal{L}(X^x_t,c)\ge-X^x_t\Phi(c)$ $\PP$-a.s.~for all $t\ge0$ and \eqref{set00} gives
\begin{align}\label{set01}
V(x,c)\ge U(x,c)-P_0+R(x,c)\qquad\text{for $x\in\RR$}
\end{align}
where
\begin{align}\label{set02}
R(x,c):=\EE\bigg[\int_0^\infty{e^{-\lambda t}\lambda\Big(P_0-X^x_t\Phi(c)\Big)dt}\bigg]\qquad\text{for $x\in\RR$}.
\end{align}
It is not hard to see from \eqref{set02} that for sufficiently negative values of $x$ we have $R(x,c)>0$ and \eqref{set01} implies that $\DD_V^c$ cannot be empty. \vspace{+4pt}

Step $2.$ \enskip Here we prove that $\DD_V^c\cap(-\infty,\gamma_*(c)]=(-\infty,\ell^{(1)}_*(c)]$ for suitable $\ell^{(1)}_*(c)\le x^0_1(c)$. The previous step has already shown that it is optimal to stop at once for sufficiently negative values of $x$. It now remains to prove that if $x\in \DD_V^c\cap(-\infty,\gamma_*(c)]$ then $x'\in \DD_V^c\cap(-\infty,\gamma_*(c)]$ for any $x'<x$. For this, fix $\bar{x}\in \DD_V^c\cap(-\infty,\gamma_*(c)]$ and let $x'<\bar{x}$. Note that the process $X^{x'}$ cannot reach a subset of $\RR$ where $\lambda P_0+\mathcal{L}(\,\cdot\,,c)<0$ (cf.~Proposition \ref{prop:discont}-$(3)$) without crossing $\bar{x}$ and hence entering $\DD_V^c$. Therefore, if $x'\in\CC_V^c$ and $\tau_*(x')$ is the associated optimal stopping time, i.e.~$\tau_*(x'):=\inf\{t\ge0\,:\,X^{x'}_t\in\DD^c_V\}$, we must have
\begin{align}\label{set05}
V(x',c)& =  U(x',c)-P_0+\EE\Big[\int_0^{\tau_*(x')} {e^{-\lambda t}\Big(\lambda P_0+\mathcal{L}(X^{x'}_t,c)\Big)dt}\Big]\ge U(x',c)-P_0,
\end{align}
 giving a contradiction and implying that $x'\in\DD_V^c$.
\vspace{+4pt}

Step $3.$ \enskip We now aim to prove that if $\DD_V^c\cap[\gamma_*(c),\infty)\neq\emptyset$ then $\DD_V^c\cap[\gamma_*(c),\infty)=[\ell_*^{(2)}(c),\ell_*^{(3)}(c)]$ for suitable $\ell_*^{(2)}(c)\le\ell_*^{(3)}(c)\le x^0_{2}(c)$. The case of $\DD_V^c\cap[\gamma_*(c),\infty)$ containing a single point is self-explanatory. We then assume that there exist $x<x'$ such that $x,x'\in\DD_V^c\cap[\gamma_*(c),\infty)$ and prove that also $[x,x']\subseteq\DD_V^c\cap[\gamma_*(c),\infty)$.

Looking for a contradiction, let us assume that there exists $y\in(x,x')$ such that $y\in\CC_V^c$. The process $X^y$ cannot reach a subset of $\RR$ where $\lambda P_0+\mathcal{L}(\,\cdot\,,c)<0$ without leaving the interval $(x,x')$ (cf.~Proposition \ref{prop:discont}-$(3)$). Then, by arguing as in \eqref{set05}, with the associated optimal stopping time $\tau_*(y):=\inf\{t\ge0\,:\,X^{y}_t\in\DD^c_V\}$, we inevitably reach a contradiction. Hence the claim follows.
\end{proof}
\vs{+4pt}

Before proceeding further we clarify the dichotomy in part $iii)$ of Proposition \ref{prop:struc-stop}, as follows. Lemma \ref{Lemma:NecessaryAndSufficientConditionSingleBoundary} below characterises the subcases $iii)(a)$ and $iii)(b)$ via condition \eqref{eq:NecessaryAndSufficientConditionSingleBoundary}. Remark \ref{rem:red} then shows that this condition does nothing more than to compare the minima of two convex functions.

\begin{lemma}\label{Lemma:NecessaryAndSufficientConditionSingleBoundary}
Fix $c \in [0,1)$ and suppose that $x^0_1(c) < \gamma_*(c) < x^0_2(c)$. 
Then $\mathcal{D}_V^c \cap [\gamma_*(c),\infty)=\emptyset$ if and only if there exists $\ell_*(c)\in(-\infty,x^0_1(c))$ such that for every $x \geq \gamma_*(c)$:
\begin{equation}\label{eq:NecessaryAndSufficientConditionSingleBoundary}
\frac{U(x,c)-P_0}{\phi_{\lambda}\left(x\right)} > \frac{U\left(\ell_*(c),c\right)-P_0} {\phi_{\lambda}\left(\ell_*(c)\right)}.
\end{equation}
\end{lemma}
\begin{proof}
$i).$ \emph{Necessity}. If $\mathcal{D}_V^c \cap [\gamma_*(c),\infty)=\emptyset$, then by Proposition \ref{prop:struc-stop}-$iii)$ there exists a point $\ell_*(c)\in(-\infty,x^0_1(c))$ such that $\mathcal{D}_V^c=(-\infty,\ell_*(c)]$. Let $x \geq \gamma_*(c)$ be arbitrary and notice that $V(x,c) < U(x,c)-P_0$ since the current hypothesis implies $x \in [\gamma_*(c),\infty) \subset \mathcal{C}^c_V$. According to Proposition \ref{prop:struc-stop}-$iii)$, the stopping time $\tau_*$ defined by
\begin{equation}\label{eq:CandidateOptimalStoppingTime}
\tau_* := \inf\{t\geq 0\,:\,X^x_t\leq \ell_*(c)\}
\end{equation}
is optimal in \eqref{dec01}. On the other hand, since $X$ has continuous sample paths and $\mathsf{P}_{x}(\{\tau_* < \infty\}) = 1$ by positive recurrence of $X$, we can also show that
\begin{align}
U(x,c)-P_0 > V(x,c) & = \mathsf{E}_{x}\bigl[e^{-\lambda \tau_*}\left(U(X_{\tau_*},c)-P_0\right)\bigr] \nonumber \\
& = \mathsf{E}_{x}\bigl[e^{-\lambda \tau_*}\left(U(\ell_*(c),c)-P_0\right)\bigr] \nonumber \\
& = \left(U(\ell_*(c),c)-P_0\right)\mathsf{E}_{x}\bigl[e^{-\lambda \tau_*}\bigr]  \nonumber \\
& = \left(U(\ell_*(c),c)-P_0\right)\frac{\phi_{\lambda}\left(x\right)}{\phi_{\lambda}\left(\ell_*(c)\right)}\label{eq:GlobalMinimumLemma-1}
\end{align}
where the last line follows from \eqref{hittingtimes}.
Since $x \geq \gamma_*(c)$ was arbitrary we have proved the necessity of the claim.
\vspace{+4pt}

$ii).$ \emph{Sufficiency}. Suppose now that there exists a point $\ell_*(c)\in(-\infty,x^0_1(c))$ such that \eqref{eq:NecessaryAndSufficientConditionSingleBoundary} holds for every $x \geq \gamma_*(c)$. Using the same arguments establishing the right-hand side of \eqref{eq:GlobalMinimumLemma-1}, noting that $\tau_*$ as defined in \eqref{eq:CandidateOptimalStoppingTime} is no longer necessarily optimal, for every $x \geq \gamma_*(c)$ we have
\begin{align*}
V(x,c) & \leq \mathsf{E}_{x}\bigl[e^{-\lambda \tau_*}\left(U(X_{\tau_*},c)-P_0\right)\bigr]\\
& = \left(U(\ell_*(c),c)-P_0\right)\frac{\phi_{\lambda}\left(x\right)}{\phi_{\lambda}\left(\ell_*(c)\right)} \\
& < U(x,c)-P_0
\end{align*}
which shows $\mathcal{D}_V^c \cap [\gamma_*(c),\infty)=\emptyset$.
\end{proof}

\begin{remark}\label{rem:red} Let us fix $c \in [0,1)$ such that $x^0_1(c) < \gamma_*(c) < x^0_2(c)$, or equivalently part $iii)$ of Proposition \ref{prop:struc-stop} holds. Writing
\begin{eqnarray}
\mathcal{F}(x)&:= & \frac{U(x,c)-P_0}{\phi_{\lambda}\left(x\right)}, \\
F(x)&:= &  \psi(x)/\phi(x),\\
H(y)&:= &  \mathcal{F}\circ F^{-1}(y) \enskip \text{for} \enskip y > 0,
\end{eqnarray}
we will appeal to the discussion given at the start of Section 6 of \cite{DaKa03}. Since $\mathcal{L}(x,c)+\lambda P_0>0$ for $x\in(-\infty,x^0_1(c))$ (from Proposition \ref{prop:discont}), it follows from equation (*) in Section 6 of \cite{DaKa03} that the function $y \mapsto H(y)$ is convex for $y\in(0,F(x^0_1(c)))$. Define $y_m^1$ and $y_m^2$ by 
\begin{equation}
\begin{split}
y_m^1 & := \arg \min \{H(y): y\in(0,F(x^0_1(c)))\} \\
y_m^2 & := \arg \min \{H(y): y\in(F(\gamma_*(c)),F(x^0_2(c)))\}
\end{split}
\end{equation}
Since the function $F$ is monotone increasing we have 
\begin{equation}
\begin{split}
m_1 & := \inf_{x \leq x^0_1(c)} \mathcal F(x)= \mathcal F(F^{-1}(y_m^1)) \\
m_2 & :=  \inf_{x \geq \gamma_*(c)} \mathcal F(x)= \mathcal F(F^{-1}(y_m^2))
\end{split}
\end{equation}

It is clear from Lemma \ref{Lemma:NecessaryAndSufficientConditionSingleBoundary} that when $m_1<m_2$ then part $iii)(a)$ of Proposition \ref{prop:struc-stop} holds, while when $m_1>m_2$ part $iii)(b)$ of Proposition \ref{prop:struc-stop} holds. (Of course, when $m_1=m_2$ then the values of $\mathcal F$ at the boundaries should also be examined to determine whether the condition of Lemma \ref{Lemma:NecessaryAndSufficientConditionSingleBoundary} holds).
\end{remark}

\subsection{The Optimal Boundaries}

We will characterise the four cases $i$, $ii$, $iii)(a)$, $iii)(b)$ of Proposition \ref{prop:struc-stop} through direct probabilistic analysis of the value function and subsequently derive equations for the optimal boundaries obtained in the previous section. We first address cases $i$ and $ii$ of Proposition \ref{prop:struc-stop}.

\begin{theorem}\label{thm:singlebd}
Let $c\in[0,1)$ and $\mathscr{B}$ be a subset of $\RR$. Consider the following problem: Find $x \in \mathscr{B}$ such that
\begin{align}\label{eq:smf00}
\big(U(x,c)-P_0\big)\frac{\phi'_\lambda(x)}{\phi_\lambda(x)}=U_x(x,c).
\end{align}
\begin{itemize}
\item[$~~i)$] If $\gamma_*(c) \ge x^0_2(c)$, let $\ell_*(c)$ be given as in Proposition \ref{prop:struc-stop}-$i)$, then $V(x,c)=V^{\ell_*}(x,c)$ (cf.~\eqref{Vl-repr}), $x\in\RR$ and $\ell_*(c)$ is the unique solution to \eqref{eq:smf00} in $\mathscr{B} = (-\infty,x^0_1(c))$.
\item[$~ii)$] If $\gamma_*(c) \le x^0_1(c)$, let $\ell_*(c)$ be given as in Proposition \ref{prop:struc-stop}-$ii)$, then $V(x,c)=V^{\ell_*}(x,c)$, $x\in\RR$ (cf.~\eqref{Vl-repr}) and $\ell_*(c)$ is the unique solution to \eqref{eq:smf00} in $\mathscr{B} =(-\infty,x^0_2(c))$.
\end{itemize}
\end{theorem}
\begin{proof}
We only provide details for the proof of $i)$ as the second part is completely analogous.

From Proposition \ref{prop:struc-stop}-$i)$ we know that $\ell_*(c)\in(-\infty,x^0_2(c))$ and that taking $\tau_*(x):=\inf\{t\ge 0\,:\,X^x_t\le\ell_*(c)\}$ is optimal for \eqref{dec01}, hence the value function $V$ is given by \eqref{Vl-repr} with $\ell=\ell_*$ (the proof is the same as that of Lemma \ref{lemma12}). If we can prove that smooth fit holds then $\ell_*$ must also be a solution to \eqref{eq:smf00}. To simplify notation set $\ell_*=\ell_*(c)$ and notice that
\begin{align}\label{prsm00}
\frac{V(\ell_*+\eps,c)-V(\ell_*,c)}{\eps}\le\frac{U(\ell_*+\eps,c)-U(\ell_*,c)}{\eps}\,,\qquad\eps>0.
\end{align}
On the other hand, consider $\tau_\eps:=\tau_*(\ell_*+\eps)=\inf\{t\ge0\,:\,X^{\ell_*+\eps}_t\le \ell_*\}$ and note that $\tau_\eps\to0$, $\PP$-a.s.~as $\eps\to0$ (which can be proved by standard arguments based on the law of iterated logarithm) 
and therefore $X^{\ell_*+\eps}_{\tau_\eps}\to\ell_*$, $\PP$-a.s.~as $\eps\to0$ by the continuity of $(t,x)\mapsto X^x_t(\omega)$ for $\omega\in\Omega$. Since $\tau_\eps$ is optimal in equation~\eqref{dec01} with $x = \ell_*+\eps$ we obtain
\begin{align}\label{prsm01}
\frac{V(\ell_*+\eps,c)-V(\ell_*,c)}{\eps}\ge\frac{\EE\Big[ e^{-\lambda\tau_\eps}\big(U(X^{\ell_*+\eps}_{\tau_\eps},c)-U(X^{\ell_*}_{\tau_\eps},c)\big)\Big]}{\eps}\,,\qquad\eps>0.
\end{align}
The mean value theorem, \eqref{OUexplicit} and \eqref{prsm01} give
\begin{align}\label{prsm02}
\frac{V(\ell_*+\eps,c)-V(\ell_*,c)}{\eps}& \ge \frac{\EE\Big[ e^{-\lambda\tau_\eps}U_x(\xi_\eps,c)\big(X^{\ell_*+\eps}_{\tau_\eps}-X^{\ell_*}_{\tau_\eps}\big)\Big]}{\eps}=\EE\Big[ e^{-(\lambda+\theta)\tau_\eps}U_x(\xi_\eps,c)\Big],
\end{align}
with $\xi_\eps\in[X^{\ell_*}_{\tau_\eps}, X^{\ell_*+\eps}_{\tau_\eps}]$, $\PP$-a.s. From \eqref{eq:Uanalyt2} one has that $U_x(\,\cdot\,,c)$ is bounded on $\RR$, hence taking limits as $\eps\to\infty$ in \eqref{prsm00} and \eqref{prsm02} and using the dominated convergence theorem in the latter we get
$V_x(\ell_*,c)=U_x(\ell_*,c)$, and since $V(\,\cdot\,,c)$ is concave (see Lemma \ref{lem:concave}) it must also be $C^1$ across $\ell_*$, i.e.~smooth fit holds.
In particular this means that differentiating \eqref{Vl-repr} at $\ell_*$ we observe that $\ell_*$ solves \eqref{eq:smf00}.
The uniqueness of this solution can be proved by the same arguments as those in part 2 of the proof of Theorem \ref{verificationOS-case2} and we omit them here for brevity.
\end{proof}

Next we address cases $iii)(a)$ and $iii)(b)$ of Proposition \ref{prop:struc-stop}. Let us define
\begin{align}\label{def:orrib}
F_1(\xi,\zeta):=\psi_{\lambda}(\xi)
\phi_{\lambda}(\zeta)\hs{-1pt}-\hs{-1pt}\psi_{\lambda}(\zeta)\phi_{\lambda}(\xi)\quad\text{and}\quad F_2(\xi,\zeta):=\psi_{\lambda}'(\xi)\phi_{\lambda}(\zeta)\hs{-2pt}-\hs{-2pt}\psi_{\lambda}(\zeta)\phi_{\lambda}'(\xi)
\end{align}
for $\xi,\zeta\in\RR$. 
\vspace{+8pt}

\begin{theorem}\label{thm:complete}
Let $c\in[0,1)$ be such that $x^0_1(c)<\gamma_*(c)<x^0_2(c)$ and consider the following problem: Find $x<y<z$ in $\RR$ with $x\in(-\infty,x^0_1(c))$ and $\gamma_*(c)<y<z<x^0_2(c)$ such that the triple $(x,y,z)$ solves the system
\begin{align}
\label{prob3i}&\quad(U(z,c)-P_0)
\frac{\phi'_{\lambda}(z)}{\phi_{\lambda}(z)}=U_x(z,c)\\[+5pt]
\label{prob3ii}&\quad(U(x,c)-P_0)\frac{F_2(x,y)}{F_1(x,y)}-(U(y,c)-P_0)
\frac{F_2(x,x)}
{F_1(x,y)}=U_x(x,c)\\[+5pt]
\label{prob3iii}&\quad(U(x,c)-P_0)\frac{F_2(y,y)}{F_1(x,y)}-(U(y,c)-P_0)\frac{F_2(y,x)}{F_1(x,y)}=U_x(y,c)
\end{align}
\begin{itemize}
\item[~i)] In case iii)(b) of Proposition \ref{prop:struc-stop} the stopping set is of the form $\DD_V^c=(-\infty,\ell_*^{(1)}(c)]\cup[\ell_*^{(2)}(c),\ell_*^{(3)}(c)]$,
and then $\{x,y,z\}=\{\ell_*^{(1)}(c),\ell_*^{(2)}(c),\ell_*^{(3)}(c)\}$ is the unique triple solving \eqref{prob3i}--\eqref{prob3iii}. The value function is given by
\begin{align}\label{eq:Vfuncomp}
\hs{-6pt}V(x,c)\hs{-2pt}=\hs{-2pt}
\left\{
\begin{array}{ll}
\hs{-4pt}\big(U(\ell^{(3)}_*,c)\hs{-2pt}-\hs{-2pt}P_0\big)\frac{\phi_\lambda(x)}{\phi_\lambda(\ell^{(3)}_*)} & \text{for $x\hs{-2pt}>\hs{-2pt} \ell^{(3)}_*$}\\[+10pt]
\hs{-4pt}U(x,c)-P_0 & \text{for $\ell^{(2)}_*\hs{-2pt}\le\hs{-2pt} x\hs{-2pt} \le\hs{-2pt} \ell^{(3)}_*$}\\[+10pt]
\hs{-4pt}(U(\ell^{(1)}_*,c)\hs{-2pt}-\hs{-2pt}P_0)\frac{F_1(x,\ell^{(2)}_*)}{F_1(\ell^{(1)}_*,\ell^{(2)}_*)}\hs{-2pt}+\hs{-2pt}(U(\ell^{(2)}_*,c)\hs{-2pt}
-\hs{-2pt}P_0)\frac{F_1(\ell^{(1)}_*,x)}
{F_1(\ell^{(1)}_*,\ell^{(2)}_*)} & \text{for $\ell^{(1)}_*\hs{-2pt}<\hs{-2pt} x\hs{-2pt} <\hs{-2pt} \ell^{(2)}_*$}\\[+10pt]
\hs{-4pt}U(x,c)-P_0 & \text{for $ x \hs{-2pt}\le\hs{-2pt} \ell^{(1)}_*$}
\end{array}
\right.
\end{align}
where we have set $\ell_*^{(k)}=\ell_*^{(k)}(c)$, $k=1,2,3$ for simplicity.
\item[ii)] In case iii)(a) of Proposition \ref{prop:struc-stop} we have $\DD_V^c=(-\infty,\ell_*^{(1)}(c)]$, 
moreover $V(x,c)=V^{\ell_*^{(1)}}(x,c)$, $x\in\RR$ (cf.~\eqref{Vl-repr}) and $\ell_*^{(1)}(c)$ is the unique solution to \eqref{eq:smf00} with $\mathscr{B}=(-\infty,x^0_1(c))$.
\end{itemize}
\end{theorem}
\begin{proof}
\emph{Proof of} $i).$ \enskip In the case of Proposition~\ref{prop:struc-stop}-$iii)(b)$, the stopping $\tau^{(II)}_*$ defined in \eqref{eq:ell*2} is optimal for \eqref{dec01}:
	\[
	V(x,c) = \EE\Big[e^{-\lambda\tau^{(II)}}\big(U(X^x_{\tau^{(II)}},c)-P_0\big)\Big]
	\]
Equation~\eqref{eq:Vfuncomp} is therefore just the analytical representation for the value function in this case. The fact that $\ell^{(1)}_*$, $\ell^{(2)}_*$ and $\ell^{(3)}_*$ solve the system \eqref{prob3i}--\eqref{prob3iii} follows from the smooth fit condition at each of the boundaries. A proof of the smooth fit condition can be carried out using probabilistic techniques as done previously for Theorem~\ref{thm:singlebd}. We therefore omit its proof and only show uniqueness of the solution to \eqref{prob3i}--\eqref{prob3iii}.

Uniqueness will be addressed with techniques similar to those employed in Theorem \ref{verificationOS-case2}, taking into account that the stopping region in the present setting is disconnected. We fix $c\in[0,1)$, assume that there exists a triple $\{\ell'_1,\ell'_2,\ell'_3\} \neq \{\ell^{(1)}_*,\ell^{(2)}_*,\ell^{(3)}_*\}$ solving \eqref{prob3i}--\eqref{prob3iii} and define a stopping time
\begin{align}\label{uniq00}
\sigma^{(II)}:=\inf\big\{t\ge0\,:\,X^x\le\ell'_1\:\:\text{or}\:\:X^x_t\in[\ell'_2,\ell'_3]\big\}\quad x\in\RR.
\end{align}
We can associate to the triple a function
\begin{align}\label{uniq01bis}
V'(x,c):=\EE\Big[e^{-\lambda\sigma^{(II)}}\big(U(X^x_{\sigma^{(II)}},c)-P_0\big)\Big]\qquad x\in\RR
\end{align}
and note that $V'(\,\cdot\,,c)$ has the same properties as the value function $V(\,\cdot\,,c)$ provided that we replace $\ell_*^{(k)}$ by $\ell'_k$ everywhere for $k=1,2,3$. Moreover, equation~\eqref{dec01} implies
\begin{equation}\label{eq:Proof-Value-Function-Triple}
V'(x,c)\ge V(x,c),\qquad x \in\RR.
\end{equation}

Step $1.$ \enskip First we show that $(\ell^{(2)}_*,\ell^{(3)}_*)\cap(\ell'_2,\ell'_3)\neq \emptyset$. We assume that $\ell'_2\ge\ell^{(3)}_*$ but the same arguments would apply if we consider $\ell^{(2)}_*\ge\ell'_3$. Note that $\ell'_1<\ell^{(3)}_*$ since $\ell'_1\in(-\infty,x^0_1(c))$, then fix $x\in(\ell'_2,\ell'_3)$ and define the stopping time $\tau_3=\inf\{t\ge 0\,:\,X^x_t\le\ell_*^{(3)}\}$. 
We have $V(x,c)<U(x,c)-P_0$ and by \eqref{uniq01bis} it follows that $V'(x,c)=U(x,c)-P_0$. Then an application of the It\^o-Tanaka formula gives
\begin{align}
0<V'(x,c)-V(x,c) = {} & \EE \Big[e^{-\lambda\tau_3}\big(V'(X^x_{\tau_3},c)-V(X^x_{\tau_3},c)\big)\Big]\\[+3pt]
&-\EE\bigg[\int_0^{\tau_3}{\hs{-4pt}e^{-\lambda t}\big(\mathbb{L}_X-\lambda\big)\big(U(X^x_t,c)-P_0\big)\mathds{1}_{\{X^x_t\in(\ell'_2,\ell'_3)\}}dt}\bigg]\nonumber\\[+3pt]
< {} & \EE\Big[e^{-\lambda\tau_3}\big(V'(\ell_*^{(3)},c)-U(\ell_*^{(3)},c)+P_0\big)\Big]\le0\nonumber
\end{align}
where we have used Proposition \ref{prop:discont}-$(3)$ in the first inequality on the right-hand side and the fact that $V'(\ell_*^{(3)},c)\le U(\ell_*^{(3)},c)-P_0 $ 
in the second. We then reach a contradiction with \eqref{eq:Proof-Value-Function-Triple} and $(\ell^{(2)}_*,\ell^{(3)}_*)\cap(\ell'_2,\ell'_3)\neq \emptyset$.
\vs{+4pt}

Step $2.$ \enskip Notice now that if we assume $\ell'_3<\ell^{(3)}_*$ we also reach a contradiction with \eqref{eq:Proof-Value-Function-Triple} as for any $x\in(\ell'_3,\ell^{(3)})_*)$ we would have $V'(x,c)<U(x,c)-P_0=V(x,c)$. Then we must have $\ell'_3\ge \ell^{(3)}_*$.

Assume now that $\ell'_3>\ell_*^{(3)}$, take $x\in(\ell_*^{(3)},\ell'_3)$ and $\tau_3$ as in Step $1.$ above. Note that $V'(x,c)=U(x,c)-P_0>V(x,c)$ whereas $V(\ell_*^{(3)},c)=U(\ell_*^{(3)},c)-P_0=V'(\ell_*^{(3)},c)$ by Step $1.$ above and \eqref{uniq01bis}. Then using the It\^o-Tanaka formula again we find
\begin{align}
0<V'(x,c)-V(x,c) = {} & \EE\Big[e^{-\lambda\tau_3}\big(V'(X^x_{\tau_3},c)-V(X^x_{\tau_3},c)\big)\Big]\\[+3pt]
&-\EE\bigg[\int_0^{\tau_3}{\hs{-4pt}e^{-\lambda t}\big(\mathbb{L}_X-\lambda\big)\big(U(X^x_t,c)-P_0\big)\mathds{1}_{\{X^x_t\in(\ell^{(3)}_*,\ell'_3)\}}dt}\bigg]\nonumber\\[+3pt]
< {} & \EE\Big[ e^{-\lambda\tau_3}\big(V'(\ell_*^{(3)},c)-U(\ell_*^{(3)},c)+P_0\big)\big]=0\nonumber
\end{align}
hence \vs{+4pt}there is a contradiction with \eqref{eq:Proof-Value-Function-Triple} and $\ell_*^{(3)}=\ell'_3$.
\vs{+4pt}

Step $3.$ \enskip If we now assume that $\ell_*^{(2)}<\ell'_2$ we find the same contradiction with \eqref{eq:Proof-Value-Function-Triple} as in Step $2.$ as in fact for any $x\in(\ell^{(2)}_*,\ell'_2)$ we would have $V'(x,c)<U(x,c)-P_0=V(x,c)$. Similarly if we assume that $\ell'_1<\ell_*^{(1)}$ then for any $x\in(\ell'_1,\ell^{(1)}_*)$ we would have $V'(x,c)<U(x,c)-P_0=V(x,c)$. These contradictions imply that it must be $\ell'_2\le \ell^{(2)}_*$ and $\ell'_1\ge \ell^{(1)}_*$.

Let us assume now that $\ell'_2<\ell^{(2)}_*$, then taking $x\in(\ell'_2,\ell^{(2)}_*)$, applying the It\^o-Tanaka formula until the first exit time from the open set $(\ell^{(1)}_*,\ell^{(2)}_*)$ and using arguments similar to those in Steps $1.$ and $2.$ we end up with a contradiction. Hence $\ell'_2=\ell^{(2)}_*$; analogous arguments can be applied to establish that $\ell'_1=\ell_*^{(1)}$.

\vs{+4pt}
\emph{Proof of} $ii).$ \enskip To prove $ii)$ we simply argue as in Theorem \ref{thm:singlebd}, concluding that $V(x,c)=V^{\ell_*^{(1)}}(x,c)$, $x\in\RR$ and $\ell_*^{(1)}$ solves \eqref{eq:smf00} with $\mathscr{B}=(-\infty,x^0_1(c))$.
\end{proof}


\section{Conclusion}\label{Sec:Conclusion}
In this paper we have studied the problem of optimal entry into an irreversible investment problem with a cost functional which is non convex with respect to the control variable. The
non convexity of the expected cost criterion is due to the real-valued nature of the spot price of electricity. We show that the problem can be decoupled and that the investment phase can be studied independently of the ``entry'' decision as an investment problem over an infinite time horizon. Instead the optimal entry decision depends heavily on the properties of the optimal investment policy. 

The complete value function can be rewritten as the one of an optimal stopping problem where the cost of immediate stopping involves the value function of the infinite horizon investment problem. It has been shown in \cite{DeAFeMo14} that the latter problem presents a complex structure of the solution, in which the optimal investment rule can be either singularly continuous or purely discontinuous, depending on the problem parameters. Such features, together with the non explicit representation of the investment problem's value function, in turn imply a non standard optimal entry policy. 
Indeed, the optimal entry rule can be either the first hitting time of the spot price at a single threshold, or can be triggered by multiple boundaries splitting the state space into non connected stopping and continuation regions. The techniques employed in the paper are those of stochastic calculus and optimal stopping theory, and a fine numerical study supports the theoretical findings of our work.
\appendix
\section{Some Proofs from Section \ref{sec:chatneg}}
\label{someproofs}
\renewcommand{\theequation}{A-\arabic{equation}}

\emph{Proof of Proposition \ref{prop-sign-LU}}\vspace{0.15cm}\\

Fix $\bar{c}\in[0,1]$ and set $\mathcal{L}(x,\bar{c}):=\big(\mathbb{L}_X-\lambda\big)U(x,\bar{c})$ for simplicity. By \eqref{eq:U&u}, \eqref{eq:LXU} and recalling that $u(x;\bar{c})=0$ for all $x\in\RR$ such that $\bar{c}\le g_*(x)$ (or equivalently $x\le\beta_*(\bar{c})$) we get
\begin{align}\label{gen-U}
\mathcal{L}(x,\bar{c})=\left\{
\begin{array}{ll}
-\lambda\,x\,\Phi(\bar{c}) & \text{for $x>\beta_*(\bar{c})$}\\[+3pt]
\big[\theta\mu-(\lambda+\theta)\,x\big]\big(g_*(x)-\bar{c}\big) -\lambda\,x\,\Phi(g_*(x)) &  \text{for $x\le \beta_*(\bar{c})$.}
\end{array}
\right.
\end{align}
Since $g_*$ is continuous with $g_*(\beta_*(\bar{c}))=\bar{c}$ one can verify that $x\mapsto\mathcal{L}(x,\bar{c})$ is continuous, $\lim_{x\to+\infty}\mathcal{L}(x,\bar{c})=-\infty$ and, by recalling also that $g_*(x)=1$ for $x\le\beta_*(1)$, $\lim_{x\to-\infty}\mathcal{L}(x,\bar{c})=+\infty$. Since $\beta_*\in C^1([0,1])$ and it is strictly monotone then $g_*$ is differentiable for a.e.~$x\in\mathbb{R}$ with $g'_*\le0$. In particular  $\tfrac{d }{d\,x}\mathcal{L}(x,\bar{c})$ exists everywhere with the exception of points $x=\beta_*(\bar{c})$ and $x=\beta_*(1)$. It follows that
\begin{align}
\label{prop-Lx01}\tfrac{d}{d\,x}\mathcal{L}(x,\bar{c})& = -\lambda\Phi(\bar{c})<0 &\text{for $x>\beta_*(\bar{c})$}\\[+4pt]
\label{prop-Lx03}\tfrac{d}{d\,x}\mathcal{L}(x,\bar{c})& = - (\lambda+\theta)\big(1-\bar{c}\big)<0 & \text{for $x<\beta_*(1)$}
\end{align}
where in the second expression we have used that $\Phi(1)=0$ by Assumption \ref{ass:phi}. Now we recall that $\beta_*(c)\le \hat{x}_0(c)$ for $c\in[0,1]$ (cf.~Proposition \ref{prop:backmat}-$i)$) and $\theta\mu-k(c)\,\beta_*(c)>0$ on $[0,1]$. Hence in particular for $c=g_*(x)$, $x\in[\beta_*(1),\beta_*(0)]$, we get $\theta\mu-k(g_*(x))\,x>0$ and
\begin{align}
\label{prop-Lx02}\tfrac{d}{d\,x}\mathcal{L}(x,\bar{c})& = \big[\theta\mu-k(g_*(x))\,x\big]g'_*(x) \\[+4pt]
&- (\lambda+\theta)\big(g_*(x)-\bar{c}\big) -\lambda\,\Phi(g_*(x))<0 & \text{for $x\in\big(\beta_*(1), \beta_*(\bar{c})\big)$.}\nonumber
\end{align}
We then obtain that $x\mapsto\mathcal{L}(x,\bar{c})$ is continuous, strictly decreasing (by \eqref{prop-Lx01}, \eqref{prop-Lx03}) and \eqref{prop-Lx02} and it equals zero at a single point for any given $\bar{c}\in[0,1]$. Obviously the result extends to $\mathcal{L}(x,\bar{c})+\lambda P_0=\big(\mathbb{L}_X-\lambda\big)\big(U(x,c)-P_0)$ and \eqref{signLU} follows.\vspace{0.2cm}\\

\noindent \emph{Proof of Proposition \ref{prop:smfit01}} \vspace{0.15cm}\\

Existence, uniqueness and smoothness of $\alpha^*$ follow from arguments analogous to those employed to prove \cite[Thm.~2.1]{DeAFeMo14}. For the limiting behaviour of $\alpha^*(c)$ as $c\to 1$ we observe that $\Phi(c)\downarrow 0$ and $x^\dagger_0(c)\uparrow\infty$ when $c\uparrow 1$. Since $\alpha^*$ is strictly increasing it has left-limit so we argue by contradiction and assume that $\alpha^*(c)\to\alpha_0$ as $c\to1$ for some $\alpha_0<+\infty$. Then in the limit as $c\to 1$ \eqref{smfit03} gives
\[
0=(\hat{G}(\alpha_0,1)-P_0)\phi'_\lambda(\alpha_0)/
\phi_\lambda(\alpha_0)=-P_0\phi'_\lambda(\alpha_0)/\phi_\lambda(\alpha_0)\neq 0\]
and we reach a contradiction. 
\vspace{0.2cm}\\

\noindent \emph{Proof of Proposition \ref{prop:upbdda}}\vspace{0.15cm}\\

Fix $c\in[0,1]$. It is clear that $\alpha^*(c)$ solves \eqref{smfit03} if and only if $K(\alpha^*(c),c)=0$ where
\begin{align}
K(x,c):=-\tfrac{\lambda}{\lambda+\theta}\Phi(c)\phi_\lambda(x)+\big(\hat{G}(x,c)-P_0\big)\phi'_\lambda(x)\qquad x\in\RR.
\end{align}
From direct computation it is not hard to verify that $x\mapsto K(x,c)$ is strictly decreasing and convex on $(-\infty,x^\dagger_0(c))$, so that it is sufficient to show that $K(z_0(c),c)<0$ for $z_0(c):=P_0/\Phi(c)$ to conclude the proof. In fact we shall only consider the case $z_0(c)\in (-\infty,x^\dagger_0(c))$ as otherwise the result is trivial.

Set for simplicity $z_0=z_0(c)$, then from straightforward algebra we find
\begin{align}\label{eq:Kz}
K(z_0,c)=\phi_\lambda(z_0(c))\Phi(c)\Big[-\tfrac{\lambda}{\lambda+\theta}+(1-\tfrac{\lambda}{\lambda+\theta})(\mu-z_0)\tfrac{\phi'_\lambda(z_0)}
{\phi_\lambda(z_0)}\Big].
\end{align}
Now, since $\phi_\lambda''>0$ and $\LL_X\phi_\lambda=\lambda\phi_\lambda$ on $\RR$ one has $(\mu-z_0)\tfrac{\phi'_\lambda(z_0)}{\phi_\lambda(z_0)}<\tfrac{\lambda}{\theta}$
hence from \eqref{eq:Kz} it follows $K(z_0,c)<0$.
\vspace{0.2cm}\\

\noindent \emph{Proof of Lemma \ref{lemma12}}\vspace{0.15cm}\\

We recall that the Ornstein-Uhlenbeck process is positively recurrent (cf.~Appendix \ref{factsOU}), hence $\tau_\ell(x,c)<+\infty$ $\PP$-a.s.~for any $x\in\RR$ and it follows that $U(X^x_{\tau_\ell},c)=U(\ell(c),c)$ $\PP$-a.s. The latter and \eqref{def-Vl} then imply
\begin{align}
V^\ell(x,c)=\big(U(\ell(c),c)-P_0\big)\EE\Big[e^{-\lambda\tau_\ell(x,c)}\Big]=\big(U(\ell(c),c)-P_0\big)\frac{\phi_\lambda(x)}{\phi_\lambda(\ell(c))}
\end{align}
for $x>\ell(c)$, where \eqref{hittingtimes} has been used.
\vspace{0.2cm}\\

\noindent \emph{Proof of Proposition \ref{prop-smfitV}}\vspace{0.15cm}\\

Fix $c\in[0,c_*)$. Since we are looking for a finite-valued boundary $\ell_*$, solving \eqref{smfit04} is equivalent to finding $x$ such that $\hat{H}(x,c)=0$ where
\begin{align}\label{def-H02}
\hat{H}(x,c):=(U(x,c)-P_0)\phi'_{\lambda}(x)-U_x(x,c)\phi_{\lambda}(x).
\end{align}
We recall \eqref{eq:U&u}, \eqref{eq:u01}, \eqref{eq:u02} and that the function $g_*$ is the inverse of $\beta_*$ (cf.~Proposition \ref{prop:backmat}). As in \eqref{eq:LXU} we can derive $U$ with respect to $x$ and take the derivative inside the integral so to obtain
\begin{align}\label{U&Ux}
U(x,c)=x(1-c)-\int^1_{c\vee g_*(x)}{\hspace{-8pt}u(x,y)dy}\quad\text{and}\quad
U_x(x,c)=(1-c)-\int^1_{c\vee g_*(x)}{\hspace{-8pt}u_x(x,y)dy}
\end{align}
for all $x\in\mathbb{R}$ and where we have used that $u$ and $u_x$ equal zero for $x\in\mathbb{R}$ such that $c\le g_*(x)$.

In order to study the asymptotic behaviour of \eqref{def-H02} as $x\to-\infty$ let us observe that for $x<b_*(1)$ one has $g_*(x)= 1$ and hence $U(x,c)=x(1-c)$ and $U_x(x,c)=1-c$. Also from the expression of $\phi_\lambda$ (cf.~Appendix \ref{factsOU}) one gets
\begin{align*}
\lim_{x \rightarrow -\infty}\phi_{\lambda}(x)= + \infty,\quad \lim_{x \rightarrow -\infty}x\phi^{'}_{\lambda}(x) = + \infty,\quad \lim_{x \rightarrow -\infty}\phi^{'}_{\lambda}(x)= - \infty
\end{align*}
and $\lim_{x \rightarrow -\infty}x\phi^{''}_{\lambda}(x) = - \infty$. Then, by applying De l'Hopital rule twice we have
\begin{align}
0\le \lim_{x \rightarrow -\infty} \frac{\phi_{\lambda}(x)}{x\phi^{'}_{\lambda}(x)}  =  \lim_{x \rightarrow -\infty} \frac{1}{1 + x\phi^{''}_{\lambda}(x)/\phi^{'}_{\lambda}(x)}  =  \lim_{x \rightarrow -\infty} \frac{1}{2 + x\phi^{'''}_{\lambda}(x)/\phi^{''}_{\lambda}(x)} \leq \frac{1}{2},
\end{align}
since $x\phi^{'''}_{\lambda}(x)/\phi^{''}_{\lambda}(x) > 0$ for $x < 0$. Therefore
$\lim_{x \rightarrow -\infty} \bigl[1 - \frac{\phi_{\lambda}(x)}{x\phi^{'}_{\lambda}(x)}\bigr] = a \in [\frac{1}{2}, 1]$ and we conclude that
\begin{align*}
\lim_{x\to-\infty}\hat{H}(x,c)& = \lim_{x\to-\infty}\left[\left(x(1-c)-P_0\right)\phi'_{\lambda}(x)-(1-c)
\phi_{\lambda}(x)\right] \\
\ge &(1-c)\lim_{x\to-\infty}x\phi'_{\lambda}(x)\left(1-\frac{\phi_{\lambda}(x)}{x\phi'_{\lambda}(x)}\right)=+\infty.
\end{align*}

Next we aim at showing that $\hat{H}(x^0(c),c)<0$ so that by continuity of $x\mapsto \hat{H}(x,c)$ we obtain existence of a solution of \eqref{smfit04}. We denote $\Sigma(c):=-\int^1_{c}{\tfrac{G(\beta_*(y),y)}{\phi_{\lambda}(\beta_*(y))}dy}>0$ where positivity holds by observing that $G(x,c)<0$ for $x<x_0(c)$ and hence for
$x=\beta_*(c)$. Then by using \eqref{def:u-analyt} and \eqref{def:G} in \eqref{U&Ux}, and evaluating the other integrals we obtain
\begin{align}
\label{U&Ux02}U(x,c)& = x\big(g_*(x)\vee c-c\big)+\tfrac{\lambda}{\lambda+\theta}\Phi\big(g_*(x)\vee c\big)\big(x+\tfrac{\mu\theta}{\lambda}\big)-\phi_{\lambda}(x)\Sigma\big(c\vee g_*(x)\big)\\[+4pt]
\label{U&Ux03}U_x(x,c)& = \big(g_*(x)\vee c-c\big)+\tfrac{\lambda}{\lambda+\theta}\Phi\big(g_*(x)\vee c\big)-\phi'_{\lambda}(x)\Sigma\big(c\vee g_*(x)\big)
\end{align}
for all $x\in\RR$.
We now substitute \eqref{U&Ux02} and \eqref{U&Ux03} inside \eqref{def-H02} to obtain
\begin{align}\label{mirac01}
\hat{H}(x,c) = {} & \left[x\big(g_*(x)\vee c-c\big)+\tfrac{\lambda}{\lambda+\theta}\Phi\big(g_*(x)\vee c\big)\big(x+\tfrac{\mu\theta}{\lambda}\big)-P_0\right]\phi'_{\lambda}(x)\nonumber\\[+4pt]
&-\left[(g_*(x)\vee c-c)+\tfrac{\lambda}{\lambda+\theta}\Phi\big(g_*(x)\vee c\big)\right]\phi_{\lambda}(x)
\end{align}
In order to evaluate \eqref{mirac01} at $x^0(c)$ we recall \eqref{gen-U} and that $x\mapsto\mathcal{L}(x,c)$ is continuous. Then it may be rewritten in a more compact form as
\begin{align}\label{LL}
\mathcal{L}(x,c)=\left[\theta\mu-(\lambda+\theta)x\right]\big(g_*(x)\vee c-c\big)-\lambda\Phi(g_*(x)\vee c) x
\end{align}
and, by definition, $x^0(c)$ is such that
\begin{align}\label{mirac02}
-P_0=\tfrac{1}{\lambda}\mathcal{L}(x^0(c),c).
\end{align}
For simplicity set $x^0:=x^0(c)$, then plugging \eqref{mirac02} into \eqref{mirac01} and using \eqref{LL} we find
\begin{align}\label{mirac03}
\hat{H}(x^0,c)=\phi_{\lambda}(x^0)\left[\frac{\theta(\mu-x^0)}{\lambda}\,\frac{\phi'_{\lambda}(x^0)}
{\phi_{\lambda}(x^0)}-1\right]\left(g_*(x^0)\vee c -c+\frac{\lambda\Phi\bigl(g_*(x^0)\vee c\bigr)}{\lambda+\theta}\right).
\end{align}
Since $\big(\mathbb{L}_X-\lambda\big)\phi_{\lambda}=0$ for all $x\in\RR$ and $\phi''_{\lambda}>0$ on $\RR$ then it holds $\theta(\mu-x)\phi'_{\lambda}(x)-\lambda\phi_{\lambda}(x)<0$ for all $x\in\mathbb{R}$. Hence from \eqref{mirac03} we obtain $\hat{H}(x^0(c),c)<0$ and there must be at least one point $\ell_*(c)<x^0(c)$ that fulfils \eqref{smfit04}. By arbitrariness of $c\in[0,c_*)$ the proof is complete.
\vspace{0.2cm}\\

\noindent \emph{Proof of Proposition \ref{prop-VxUx}}\vspace{0.15cm}\\

First fix an arbitrary $c\in[0,1]$ and recall \eqref{eq:U&u}, \eqref{eq:u01} and \eqref{eq:u02}. Then by standard arguments based on dominated convergence theorem we get
\begin{align}
U_{xx}(x,c)=\phi''_{\lambda}(x)\int_{g_*(x)\vee c}^1 \frac{G({\beta_*(y)},y)}{\phi_{\lambda}({\beta_*(y)})}dy\qquad x\in\RR
\end{align}
by the affine nature of $x\mapsto G(x,c)$ (cf.~\eqref{def:G}). As expected $U_{xx}(\,\cdot\,,c)$ is continuous on $\RR$. Now by differentiating separately in the two regions $\big\{x\in\RR\,:\,c>g_*(x)\big\}$ and $\big\{x\in\RR\,:\,c< g_*(x)\big\}$, with the exception of points $x=\beta_*(1)$ and $x=\beta_*(0)$, recalling that $g_*$ is $C^1$ elsewhere (cf.~Proposition \ref{prop:backmat}-$i)$), $g_*'=0$ on $(-\infty,\beta_*(1))\cup(\beta_*(0),+\infty)$ and $\beta_*(g_*(x))=x$, we find
\begin{align}\label{exprFxxx}
U_{xxx}(x,c)=\phi'''_{\lambda}(x)\int_{g_*(x)\vee c}^1 \frac{G({\beta_*(y)},y)}{\phi_{\lambda}({\beta_*(y)})}dy -\phi''_{\lambda}(x)\frac{G(x,g_*(x))}{\phi_{\lambda}(x)}\,g'_*(x)\mathds{1}_{\{c< g_*(x)\}}
\end{align}
for a.e.~$x\in\mathbb{R}$ which shows
\begin{equation}\label{lem:Uxxx}
\forall c \in [0,1]: \enskip U_{xxx}(\,\cdot\,,c)\in L^\infty_{loc}(\RR).
\end{equation}

Now fix $\bar{c}\in[0,c_*)$ and take $\ell_*(\bar{c})$ solving \eqref{smfit04}. Since by definition $V^{\ell_*}(\ell_*(\bar{c}),\bar{c})=U(\ell_*(\bar{c}),\bar{c})-P_0$ it suffices to show that $V^{\ell_*}_x(\,\cdot\,,\bar{c})\le U_x(\,\cdot\,,\bar{c})$ on $\RR$ to verify the claim. The latter trivially holds for $x\le \ell_*(\bar{c})$ by \eqref{Vl-repr} and \eqref{smfit04}, hence it remains to prove it for $x>\ell_*(\bar{c})$.

From \eqref{Vl-repr} it follows that $\mathbb{L}_XV^{\ell_*}(x,\bar{c})-\lambda V^{\ell_*}(x,\bar{c})=0$ for $x>\ell_*(\bar{c})$ and it is not hard to verify by direct derivation of the latter that
\begin{align}
\mathbb{L}_XV_x^{\ell_*}(x,\bar{c})-(\lambda+\theta) V_x^{\ell_*}(x,\bar{c})=0\quad\text{for $x>\ell_*(\bar{c})$}
\end{align}
as well. On the other hand, from \eqref{gen-U} one obtains that $\mathcal{L}(\,\cdot\,,\bar{c})$ is differentiable for a.e.~$x\in\RR$, in particular with the exception of $x=\beta_*(\bar{c})$ and $x=\beta_*(1)$ (the latter by non differentiability of $g_*$ at that point). Then by \eqref{lem:Uxxx}, we obtain $\mathcal{L}_x(x,\bar{c})=\left(\mathbb{L}_X-(\lambda+\theta)\right)U_x(x,\bar{c})$ a.e.~$x\in\RR$ and with
\begin{align}\label{gen-Ux}
\mathcal{L}_x(x,\bar{c}):=\left\{
\begin{array}{ll}
-\lambda\,\Phi(\bar{c}) & \text{for $x>\beta_*(\bar{c})$}\\[+5pt]
-\lambda\Phi(g_*(x))-(\lambda+\theta)(g_*(x)-\bar{c}) & \\[+3pt]
\hspace{+60pt}+\big[\theta\mu-k\big(g_*(x)\big)\,x\big]g'_*(x) &  \text{for a.e.~$x\le \beta_*(\bar{c})$.}
\end{array}
\right.
\end{align}
Notice that since $\beta_*$ is strictly decreasing, $g'_*$ is bounded on $\RR$ and $\mathcal{L}_x(\,\cdot\,,\bar{c})$ is locally bounded on $\RR$ with $\big|\mathcal{L}_x(\,\cdot\,,\bar{c})\big|\le C(1+|x|)$ for $x\in\RR$ and a suitable constant $C>0$.

Define
\begin{align}\label{def:tau-l*}
\tau_\ell^*(x,\bar{c}):=\inf\{t\ge0\,\:\,X^x_t\le\ell_*(\bar{c})\}\qquad x\in\RR,
\end{align}
fix $x>\ell_*(\bar{c})$ and denote $\tau^*_\ell=\tau_\ell^*(x,\bar{c})$ for simplicity. Take $R>0$ arbitrary and fixed such that $-R<\beta_*(1)$ and $R>\beta_*(0)$, and denote $\tau_R:=\inf\{t\ge 0\,:\,|X^x_t|\ge R\}$. Since $U_{xx}(\cdot,\bar{c})$ is continuous and $U_{xxx}(\cdot,\bar{c})$ locally bounded, then we use an extension of It\^o's formula based on preliminary mollification of $U_x$ (cf.~\cite[Ch.~8,~Sec.~VIII.4,~Thm.~4.1]{FlemingSoner}) to obtain
\begin{align}
V^{\ell_*}_x(x,\bar{c})-U_x(x,\bar{c})
= {} & \EE\left[e^{-(\lambda+\theta)(\tau_{\ell}^*\wedge\tau_R)}\left(V^{\ell_*}_x\big(X^x_{\tau_{\ell}^*\wedge\tau_R},\bar{c}\big)-
U_x\big(X^x_{\tau_{\ell}^*\wedge\tau_R},\bar{c}\big)\right)\right] \nonumber \\
&+\EE\left[\int^{\tau_{\ell}^*\wedge\tau_R}_0{e^{-(\lambda+\theta)s}
\mathcal{L}_x\big(X^x_{s},\bar{c})ds}\right]\nonumber\\
\le {} & \EE\left[e^{-(\lambda+\theta)(\tau_{\ell}^*\wedge\tau_R)}\left(V^{\ell_*}_x\big(X^x_{\tau_{\ell}^*\wedge\tau_R},\bar{c}\big)-
U_x\big(X^x_{\tau_{\ell}^*\wedge\tau_R},\bar{c}\big)\right)\right] \label{PDE}
\end{align}
where the inequality is due to \eqref{prop-Lx01}, \eqref{prop-Lx03} and \eqref{prop-Lx02}. In order to evaluate the last expression in the right-hand side of \eqref{PDE} notice that on the set $\{\tau_R<\tau^*_\ell\}$ one has either $X^x_{\tau_R}=-R$ which implies $V^{\ell_*}_x\big(-R,\bar{c}\big)-
U_x\big(-R,\bar{c}\big)=0$ by \eqref{Vl-repr}, or $X^x_{\tau_R}=R$ which, along with \eqref{Vl-repr} and \eqref{U&Ux03}, implies instead $\big|V^{\ell_*}_x\big(R,\bar{c}\big)-
U_x\big(R,\bar{c}\big)\big|\le D (1+|\phi_\lambda'(R)|)$ for a suitable constant $D>0$. Hence, noting that $V^{\ell_*}_x\big(X^x_{\tau_{\ell}^*},\bar{c}\big)-
U_x\big(X^x_{\tau_{\ell}^*},\bar{c}\big)=0$, $\PP$-a.s.~by the smooth fit condition \eqref{smfit04}, we get
\begin{align}\label{eq:Prop-4-13-Proof}
V^{\ell_*}_x(x,\bar{c})-U_x(x,\bar{c})\le D (1+|\phi_\lambda'(R)|) \EE\bigl[e^{-(\lambda+\theta)\tau_R}\mathds{1}_{\{\tau_R<\tau^*_\ell\}}\mathds{1}_{\{X^x_{\tau_R}=R\}}\bigr] \quad\text{for $x>\ell_*(\bar{c})$}
\end{align}
Since $R \mapsto \phi_\lambda(R)$ is strictly convex and decreasing (cf. Section~\ref{factsOU}), the function $R \mapsto \phi_\lambda'(R)$ is negative and increasing, which means $R \mapsto D (1+|\phi_\lambda'(R)|)$ is non-negative and \emph{decreasing}. By taking limits as $R\to\infty$ in \eqref{eq:Prop-4-13-Proof}, and recalling also the discussion above, we conclude that $V^{\ell_*}_x(x,\bar{c})-U_x(x,\bar{c})\le0$ for $x\in\RR$. The proof is complete since $\bar{c}$ was arbitrary.

\section{Some Facts on the Ornstein-Uhlenbeck Process}
\label{factsOU}
\renewcommand{\theequation}{B-\arabic{equation}}

Recall the Ornstein-Uhlenbeck process $X$ of \eqref{OU}. It is well known that $X$ is a positively recurrent Gaussian process (cf., e.g., \cite{BorodinSalminen}, Appendix 1, Section 24, pp.\ 136-137) with state space $\mathbb{R}$ and that \eqref{OU} admits the explicit solution
\beq
\label{OUexplicit}
X^x_t= \mu + (x-\mu)e^{-\theta t} + \int_0^t \sigma e^{\theta(s-t)}dB_s.
\eeq
We introduced its infinitesimal generator $\mathbb{L}_{X}$ in \eqref{def:LX};
the characteristic equation $\mathbb{L}_{X}u = \lambda u$, $\lambda > 0$, admits the two linearly independent, positive solutions (cf.~\cite{JYC}, p.\ 280)
\beq
\label{phi}
\phi_{\lambda}(x):=
e^{\frac{\theta(x-\mu)^2}{2\sigma^2}}D_{-\frac{\lambda}{\theta}}\Big(\frac{(x-\mu)}{\sigma}\sqrt{2\theta}\Big)
\eeq
and
\beq
\label{psi}
\psi_{\lambda}(x):=
e^{\frac{\theta(x-\mu)^2}{2\sigma^2}}D_{-\frac{\lambda}{\theta}}\Big(-\frac{(x-\mu)}{\sigma}\sqrt{2\theta}\Big),
\eeq
which are strictly decreasing and strictly increasing, respectively. In both \eqref{phi} and \eqref{psi} $D_{\alpha}$ is the cylinder function of order $\alpha$ (see \cite{Trascendental}, Chapter VIII, among others) and it is also worth recalling that (see, e.g., \cite{Trascendental}, Chapter VIII, Section 8.3, eq.\ (3) at page 119)
\begin{align}
\label{cylinder}
D_{\alpha}(x):= \frac{e^{-\frac{x^2}{4}}}{\Gamma(-\alpha)}\int_0^{\infty}t^{-\alpha -1} e^{-\frac{t^2}{2} - x t} dt, \quad \text{Re}(\alpha)<0,
\end{align}
where $\Gamma(\cdot)$ is Euler's Gamma function.

We denote by $\PP_x$ the probability measure on $(\Omega, \cF)$ induced by the process $(X^x_t)_{t\ge0}$, i.e.~such that $\mathbb{P}_x(\,\cdot\,) = \mathbb{P}(\,\cdot\,| X(0)=x)$, $x \in \mathbb{R}$, and by $\EE_x[\,\cdot\,]$ the expectation under this measure.  Then, it is a well known result on one-dimensional regular diffusion processes (see, e.g., \cite{BorodinSalminen}, Chapter I, Section 10) that
\begin{equation}
\label{hittingtimes}
\EE_x[e^{-\lambda\tau_y}]=
\left\{
\begin{array}{ll}
\displaystyle \frac{\phi_{\lambda}(x)}{\phi_{\lambda}(y)}, \quad x \geq y,\\
\\
\displaystyle \frac{\psi_{\lambda}(x)}{\psi_{\lambda}(y)}, \quad x \leq y,
\end{array}
\right.
\end{equation}
with $\phi_{\lambda}$ and $\psi_{\lambda}$ as in \eqref{phi} and \eqref{psi} and $\tau_{y}:=\inf\{t \geq 0: X^x_t = y\}$ the hitting time of $X^x$ at level $y \in \mathbb{R}$. Due to the recurrence property of the Ornstein-Uhlenbeck process $X$ one has $\tau_{y} < \infty$ $\PP_x$-a.s.\ for any $x,y \in \R$.

It is also useful to recall here some convergence and integrability properties of $X$.
\begin{lemma}
\label{lem:convX}
One has 
$$\liminf_{t \uparrow \infty}e^{-\lambda t}|X^x_t| =0, \qquad a.s.$$
\end{lemma}
\begin{proof}
Define $\Xi:=\liminf_{t \uparrow \infty}e^{-\lambda t}|X^x_t|$ and notice that clearly $\Xi\geq 0$ a.s.
We now claim (and prove later) that $\liminf_{t \uparrow \infty}e^{-\lambda t}\EE\big[|X^x_t|\big]=0$ to obtain by Fatou Lemma
\begin{equation}
\label{FatouconvX}
0 \leq \EE\big[\Xi\big] \leq \liminf_{t \uparrow \infty}e^{-\lambda t}\EE\big[|X^x_t|\big]= 0;
\end{equation}
that is, $\EE\big[\Xi\big]=0$ and hence $\Xi=0$ a.s.\ by nonnegativity of $\Xi$.

To complete the proof we have thus only to show that $\liminf_{t \uparrow \infty}e^{-\lambda t}\EE\big[|X^x_t|\big]=0$.
By \eqref{OUexplicit} and H\"{o}lder inequality one has 
\begin{align}
\label{convEX}
\EE\big[|X^x_t|\big] & \leq \mu + e^{-\theta t}|x-\mu| + e^{-\theta t}\EE\bigg[\Big|\int_0^t e^{\theta s}dB_s\Big|^2\bigg]^{\frac{1}{2}} \\
& = \mu + e^{-\theta t}|x-\mu| + e^{-\theta t}\frac{1}{2\theta}(e^{2\theta}-1)^{\frac{1}{2}}, \nonumber 
\end{align}
where also It\^o isometry has been used. It is now easily checked that \eqref{convEX} implies the claim.
\end{proof}

\begin{lemma}\label{lem:uiX}
Fix $x\in\RR$, and set $\tau_R:=\inf\{t\ge0\,:\,|X^x_t|\ge R\}$, $R>0$,
then the family
$\{e^{-\lambda\tau_R}|X^x_{\tau_R}|\,:\,R>0 \}$ is uniformly integrable.
\end{lemma}
\begin{proof}
It suffices to show that $\{e^{-\lambda\tau_R}|X^x_{\tau_R}|\,:\,R>0 \}$ is uniformly bounded in $L^2(\Omega,\PP)$. With no loss of generality we take $x\in (-R,R)$ so that we can write $\tau_R=\tau^+_R\wedge\tau^-_R$ $\PP$-a.s.~with $\tau^+_R:=\inf\{t\ge0\,:\,X^x_t\ge R\}$ and $\tau^-_R:=\inf\{t\ge0\,:\,X^x_t\le -R\}$. From recurrence of $X$ we get
\begin{align}\label{eq:Xui1}
\EE\Big[ e^{-2\lambda \tau_R}|X^x_{\tau_R}|^2\Big]& = R^2\Big(\EE\Big[e^{-2\lambda\tau^+_R}\mathds{1}_{\{\tau^+_R<\tau^-_R\}}\Big]+\EE  \Big[e^{-2\lambda\tau^-_R}\mathds{1}_{\{\tau^-_R<\tau^+_R\}}\Big] \Big)\\
& \le  R^2\left[\tfrac{\psi_{2\lambda}(x)}{\psi_{2\lambda}(R)}+\tfrac{\phi_{2\lambda}(x)}{\phi_{2\lambda}(-R)}\right]\nonumber.
\end{align}
As $R\to \infty$ the functions $\phi_{2\lambda}(-R)$ and $\psi_{2\lambda}(R)$ diverge to infinity with a super quadratic trend, hence there exists a constant $C(x)>0$ depending only on $x\in\RR$ such that $$\sup_{R>0}\EE\Big[e^{-2\lambda \tau_R}|X^x_{\tau_R}|^2\Big]\le C(x).$$
\end{proof}

\bigskip

\textbf{Acknowledgements.}
The first, third and fourth named authors express their gratitude to the UK Engineering and Physical Sciences Research Council (EPSRC) for its financial support via grant EP/K00557X/1. Financial support by the German Research Foundation (DFG) via grant Ri--1128--4--2 is gratefully acknowledged by the
second named author.

\newpage

\par \leftskip=24pt

\noindent Tiziano De Angelis \\
School of Mathematics \\
University of Leeds \\
Woodhouse Lane \\
Leeds LS2 9JT 9PL \\
United Kingdom \\
\texttt{t.deangelis@leeds.ac.uk}

\vspace{2pc}

\par \leftskip=24pt

\noindent Giorgio Ferrari \\
Center for Mathematical Economics \\
Bielefeld University \\
Universit\"atsstrasse 25 \\
D-33615 Bielefeld \\
Germany\\
\texttt{giorgio.ferrari@uni-bielefeld.de}

\vspace{2pc}

\par \leftskip=24pt

\noindent Randall Martyr \\
School of Mathematical Sciences \\
Queen Mary University of London \\
Mile End Road \\
London E1 4NS \\
United Kingdom \\
\texttt{r.martyr@qmul.ac.uk}
\vspace{2pc}

\par \leftskip=24pt

\noindent John Moriarty \\
School of Mathematical Sciences \\
Queen Mary University of London \\
Mile End Road \\
London E1 4NS \\
United Kingdom \\
\texttt{j.moriarty@qmul.ac.uk}
\end{document}